\documentclass[11pt]{article}
\usepackage{amsmath,amssymb,amsthm,color,comment,geometry,graphics,hyperref}
\usepackage[T1]{fontenc}
\usepackage[utf8]{inputenc}
\usepackage{authblk}

\newtheorem{theorem}{Theorem}[section]
\newtheorem{lemma}{Lemma}[section]
\newtheorem{proposition}{Proposition}[section]

\theoremstyle{definition}
\newtheorem{definition}{Definition}[section]
\theoremstyle{remark}
\newtheorem{remark}{Remark}[section]

\newcommand{\Tr}{{\rm Tr}}
\newcommand{\tr}{{\rm tr}}
\newcommand{\ricden}{{\bf Ric}}
\newcommand{\ricfun}{\mathcal{R}ic}
\newcommand{\ric}{{\rm Ric}}
\newcommand{\riem}{{\rm Riem}}
\newcommand{\riemop}{ Riem}
\newcommand{\scalar}{R}
\newcommand{\scalarfun}{\mathcal{R}}
\newcommand{\res}{\rm Res}

\newcommand{\pr}{{ Q}}
\newcommand{\A}{A_\theta}
\newcommand{\Ai}{A_\theta^\infty}
\newcommand{\Hcal}{\mathcal{H}}
\newcommand{\End}{{\rm End}}
\newcommand{\lap}{\triangle}
\newcommand{\conn}{\bigtriangledown }
\newcommand{\compcent}[1]{\vcenter{\hbox{$#1\circ$}}}
\newcommand{\comp}{\mathbin{\mathchoice
{\compcent\scriptstyle}{\compcent\scriptstyle}
{\compcent\scriptscriptstyle}{\compcent\scriptscriptstyle}}} 
\makeatletter
\newcommand{\extp}{\@ifnextchar^\@extp{\@extp^{\,}}}
\def\@extp^#1{\mathop{\bigwedge\nolimits^{\!#1}}}
\makeatother

\title{The Ricci Curvature in Noncommutative Geometry}

\author[$\dag$]{Remus Floricel}
\author[$\dag$]{Asghar Ghorbanpour}
\author[$\ddag$]{Masoud Khalkhali \thanks{R. Floricel and M. Khalkhali are supported by NSERC Discovery grants. A. Ghorbanpour is partially supported by a PIMS postdoctoral fellowship.}}
\affil[$\dag$]{Department of Mathematics and    Statistics, 
University of Regina}
\affil[$\ddag$]{Department of Mathematics, University of Western Ontario}

\newcommand{\Addresses}{{
  \bigskip
  \footnotesize

  Remus Floricel, \textsc{Department of Mathematics and Statistics,%
University of Regina, 
Regina, SK, Canada, S4S 0A2}\par\nopagebreak
  \textit{E-mail address}: \texttt{Remus.Floricel@uregina.ca}

  \medskip

  Asghar Ghorbanpour,\textsc{Department of Mathematics and Statistics,%
University of Regina, 
Regina, SK,  Canada, S4S 0A2}\par\nopagebreak
  \textit{E-mail address}:\texttt{asghar.ghorbanpour@uregina.ca}

  \medskip

 Masoud Khalkhali, \textsc{Department of Mathematics, University of Western Ontario,
London, Ontario, Canada, N6A 5B7}\par\nopagebreak
  \textit{E-mail address}: \texttt{masoud@uwo.ca}

}}

\date{ }
\begin{document}
\maketitle

\abstract{Motivated by the  local formulae for asymptotic expansion of heat kernels in spectral geometry, we propose a definition of Ricci curvature in noncommutative settings.   The Ricci operator of an oriented closed Riemannian manifold can be realized
as a spectral functional, namely the functional defined by the zeta
function of the full Laplacian of the de Rham complex, localized by smooth endomorphisms  of the cotangent bundle and their trace.
We  use this formulation to introduce the Ricci functional in a  noncommutative setting and in particular for curved noncommutative tori. This Ricci functional uniquely determines a density element, called the Ricci density, which plays the role of the Ricci operator. 
The main result of this paper provides an explicit computation of the Ricci density when the conformally flat geometry of the noncommutative two torus is encoded by the modular de Rham spectral triple.  }
\tableofcontents
\allowdisplaybreaks

\section{Introduction}
 In Connes' program of noncommutative geometry \cite{Connes1994,Connes-Marcolli2008}, the role of geometrical objects is played by spectral triples $(\mathcal{A},\mathcal{H},D)$. 
Similar to the  commutative case and the standard spectral triple $(C^\infty(M),L^2(S),D)$, where $(M,g,S)$ is a closed spin manifold  and $D$ is the Dirac operator acting on the spinor bundle $S$, the spectrum of the Dirac operator $D$ of a spectral triple $(\mathcal{A},\mathcal{H},D)$ encodes the geometrical information of the spectral triple.
However, to gain access to this information, one should first find a spectral formulation of the specific geometric notion, and then extend it to the level of spectral triples.
For instance, the dimension of the manifold is captured by the notion of $p+$ summability, 
 and integration with respect to the Riemannian volume form is captured by the Dixmier trace \cite{Connes1995-reality}.

The notion of scalar curvature for spectral triples has also been formulated in this manner \cite{Connes-Marcolli2008} as we recall now. 
Let $(\mathcal{A,H},D)$ be a spectral triple  of metric dimension $m$ whose (localized) trace of heat kernel has an  asymptotic expansion of the form
\begin{equation}\label{heattraceasymcondition}
\Tr(ae^{-tD^2})\sim\sum_{n=0}^\infty a_n(a,D^2) t^{\frac{n-m}{2}},\quad a\in \mathcal{A},
\end{equation} as $t\to 0^+$.
The scalar curvature is then represented by  the scalar curvature  functional on $\mathcal{A}$, given by
 \begin{equation*}
\scalarfun(a)=a_2(a,D^2).
 \end{equation*}
 This functional can be written in terms of the (localized) spectral zeta function of $D^2$, $\zeta_{D^2}(a,z)=\Tr\big(aD^{-2z}({\rm I}-\pr)\big)$, for $\Re z>m/2$. Here $\pr$ denotes the orthogonal projection on the kernel of  $D^2$, and $a\in \mathcal{A}$. 
Using the Mellin transform, we then have \cite{Connes-Marcolli2008}:
\begin{equation*}
\scalarfun(a)=\left\{
\begin{array}{lc}
\zeta_{D^2}(a,0)-\Tr(aQ), & \mbox{if}\;\,m=2\\ &\\ 
{\res}_{s=\frac{m}{2}-1}\zeta_{D^2}(a,s), &\mbox{if}\;\, m>2.
\end{array}
\right.
\end{equation*} 
This definition is motivated by the classical case $(C^\infty(M),L^2(S),D)$ \cite[Theorem 1.148]{Connes-Marcolli2008},
 where it was shown that
\begin{equation*}
\scalarfun(f)=C_m\int f(x)\scalar(x)dx,\quad f\in C^\infty(M).
\end{equation*} 
Here $R$ is the the scalar curvature of the metric $g$ and $C_m$ is a constant that depends only on the dimension $m$ of the manifold. 

Note that in the commutative  case, the functional $\scalarfun$ determines uniquely the scalar curvature $R$, as the density function of  $\scalarfun$. 
Although the scalar curvature functional was defined in \cite{Connes-Marcolli2008} in the general case of a spectral triple satisfying \eqref{heattraceasymcondition}, an  explicit computation of the scalar curvature $R$ for  the noncommutative torus $A_\theta$ is a  formidable task that required  intriguing  analytic ideas and computer assistance \cite{Connes-Moscovici2014,Fathizadeh-Khalkhali2013}. The quest to prove the Gauss-Bonnet theorem for curved noncommutative tori in the pioneering work of Connes and Tretkoff \cite{Connes-Cohen1992, Connes-Tretkoff2011} played a fundamental role here. As a rule, in all calculations involving the noncommutative tori,   Connes' pseudodifferential calculus \cite{Connes1980} plays a fundamental role.

Unlike the scalar curvature, the Ricci curvature does not appear in the coefficients of the heat trace of the Dirac Laplacian, $D^2$. 
Nevertheless, the Laplacian of the de Rham complex, more precisely the Laplacian on one forms, captures the Ricci operator in its second term.
Exploiting this observation, we formulate the  Ricci operator as a spectral functional on  the algebra of sections of the endomorphism bundle of the cotangent bundle of $M$ (Definition \ref{riccifunctional}), and call it the Ricci functional:
\begin{equation*}
\ricfun(F)=a_2(\tr(F),\lap_0)-a_2(F,\lap_1),\quad F\in C^\infty(\End(T^*M)).
\end{equation*}
An equivalent version of the Ricci functional in terms of the spectral zeta function is then given in \S \ref{zetaformulationofricci}:
\begin{equation*}
\ricfun(F)=\begin{cases}
\zeta(0,\tr(F),\lap_0)-\zeta(0,F,\lap_1)+\Tr(\tr(F)\pr_0)-\Tr(FQ_1), & m=2\\ &\\ 
 \Gamma(\frac{m}{2}-1)\res_{s=\frac{m}{2}-1}\Big(\zeta(s,\tr(F),\lap_0)-\zeta(s,F,\lap_1)\Big), & m>2,
\end{cases}
\end{equation*}where $\pr_j$ is the orthogonal projection on the kernel of $\lap_j$. 

In order to define the Ricci functional for the curved noncommutative two torus, we first construct the analogue of the de Rham complex for the noncommutative two torus. For this purpose, following \cite{Connes-Cohen1992, Connes-Tretkoff2011}, we conformally change the metric by using a noncommutative Weyl factor $e^{-h}$ with $h=h^*\in \Ai$. This procedure gives rise to the modular de Rham spectral triple with dilaton $h$ (Definition \ref{RiccifundeRahm}), which is a modular spectral triple in the sense of \cite{Connes-Moscovici2008}. 
We then  define the Ricci functional for the modular de Rham spectral triple, and show that there exists an element $\ricden\in \Ai\otimes M_2(\mathbb{C})$, called the Ricci density, such that (Lemma \ref{Riccidensitylemma}), 
\begin{equation*}
\ricfun(F)=\frac{1}{\Im(\tau)}\varphi(\tr(F\ricden)e^{-h}), \quad F\in \Ai\otimes M_2(\mathbb{C}).
\end{equation*}
The main result of the paper, obtained in Theorem \ref{maintheorem}, provides a thorough description of the Ricci density:
\begin{equation*}
\ricden= \frac{\Im(\tau)}{4\pi^2}R^\gamma\otimes {\rm I}_2- \frac{1}{4\pi} S(\nabla_1,\nabla_2)\big([\delta_1(\log k),\delta_2(\log k)]\big)e^h\otimes \begin{pmatrix}
i\Im(\tau) & \Im(\tau)^2  \\
 -1& i\Im(\tau)
\end{pmatrix}\,.
\end{equation*}
The term  $R^\gamma\in\Ai$ turns out to be equal to  the graded scalar curvature computed in \cite{Connes-Moscovici2014,Fathizadeh-Khalkhali2013}, $\nabla(a)=-[h,a]$, and the function $S$ is given by   
\begin{equation*}
S(s,t)=\frac{(s+t-t\, \cosh(s)-s\, \cosh(t)-\sinh(s)-\sinh(t)+\sinh(s+t))}{s\, t\left(\sinh\left(\frac{s}{2}\right) \sinh\left(\frac{t}{2}\right) \sinh\left(\frac{s+t}{2}\right)\right)},
\end{equation*}
which coincides with the function $S$ found in \cite{Connes-Moscovici2014,Fathizadeh-Khalkhali2013} for scalar curvature. 
The computations and the proof of the theorem are placed in the last section \S\ref{computation}.

It is an interesting feature of noncommutative geometry that, contrary to the commutative case, the Ricci curvature is not a multiple of the scalar curvature  even in dimension two. This manifests itself in the existence of off diagonal terms 
in the Ricci operator $\ricden $ above.

It is clear that one can define in a similar fashion a Ricci curvature operator for higher dimensional noncommutative tori, as well as for noncommutative toric manifolds. Its computation in these cases poses an interesting   problem. It would also be interesting to find the analogue of the Ricci flow based on our definition of Ricci curvature functional. It should be noted that for noncommutative two tori a  definition of Ricci flow, without a notion of Ricci curvature, is proposed in \cite{Tanvir-Marcolli2012}.

The spectral geometry of a curved  noncommutative two torus has been the subject of intensive studies in recent years. Starting with the pioneering work \cite{Connes-Cohen1992}, 
a Gauss-Bonnet theorem is proven in \cite{Connes-Tretkoff2011} and for general conformal structures in \cite{Fathizadeh-Khalkhali2012}, while  the scalar curvature for conformally flat metrics is computed in \cite{Connes-Moscovici2014,Fathizadeh-Khalkhali2013}.  This scalar curvature and its relation to higher order terms in the heat kernel expanion is further studied in \cite{Connes-Fathizadeh2016}.     A  version of the Riemann-Roch theorem is proven in \cite{Moatadelro-Khalkhali2014} and in general in \cite{Lesch-Moscovici2016}.
 The key idea in all of these works is that 
 the conformal change of metric, first introduced in  \cite{Connes-Cohen1992, Connes-Tretkoff2011}, can be implemented in the noncommutative two torus by introducing a noncommutative Weyl factor. 
The complex geometry of the noncommutative two torus, on the other hand, provides a Dirac operator which, in analogy with the classical case, originates from the Dolbeault complex. 
By perturbing this spectral triple, one can construct a (modular) spectral triple that can be used to study the geometry of the conformally perturbed flat metric on the noncommutative two torus.  
Then, using the pseudodifferrential operator theory for $C^\ast$-dynamical systems developed in \cite{Connes1980}, the computation is performed and explicit formulas are obtained. The spectral geometry and study of scalar curvature  of  noncommutative tori has been pursued further in  
\cite{ Dabrowski-Sitarz2015,Fathi-Khalkhali2015,Fathi-Ghorbanpour-Khalkhali2016,Fathizadeh-Khalkhli2013-2,Khalkhali-Motadelro-Sadeghi2106,Fathizadeh2015,Liu2015}.

\section{Ricci functional on Riemannian manifolds}\label{classicalcase}
In this section, a spectral definition for the Ricci curvature is provided in the classical case.
Let $(M,g)$ be an oriented, closed Riemannian manifold of dimension $m$. 
We will follow the convention used in \cite{Gilkey2004} for the curvature tensor, however we will fix our own notation.
Let $\conn$ be the Levi-Civita connection of the metric $g$.
The Riemannian operator, $\riemop$, and the curvature tensor, $\riem$, are define by
\begin{eqnarray*}
\riemop(X,Y)&:=&\conn_X\conn_Y-\conn_Y\conn_X-\conn_{[X,Y]},\\
\riem(X,Y,Z,W)&:=&g(\riemop(X,Y)Z,W).
\end{eqnarray*}
With respect to the coordinate frame  $\partial_\mu=\frac{\partial}{\partial x^\mu}$, the components 
of the curvature tensor are denoted by
$$\riem_{\mu\nu\rho\epsilon}:=\riem(\partial_\mu,\partial_\nu,\partial_\rho,\partial_\epsilon).$$
The components of  the Ricci tensor $Ric$ and scalar curvature $R$ are given by
 \begin{eqnarray*}
\ric_{\mu\nu}&:=&g^{\rho\epsilon}\riem_{\mu\rho\epsilon\nu},\\
\scalar&:=&g^{\mu\nu}\ric_{\mu\nu}=g^{\mu\nu}g^{\rho\epsilon}\riem_{\mu\rho\epsilon\nu}.
\end{eqnarray*}

\subsection{Ricci curvature as a spectral functional of the de Rham complex}
Let $P:C^\infty(V)\to C^\infty(V)$ be a positive elliptic differential operator of order two  acting on the sections of a smooth Hermitian vector bundle $V$ over $M$.
The heat trace $\Tr(e^{-tP})$ admits a complete asymptotic expansion of the form \cite[Lemma 1.8.2]{Gilkey1995}  
\begin{eqnarray*}
\Tr(e^{-tP})\sim \sum_{n=0}^\infty a_n(P)t^{\frac{n-m}2},\qquad t\to 0^+.
\end{eqnarray*}
Each coefficient $a_n(P)$ can be realized as the integral of the trace of a locally computable  $\End(V)$-valued density $a_n(x,P)$;
\begin{equation*}
a_n(P)=\int \tr(a_n(x,P))dx.
\end{equation*}
Here $dx=\sqrt{\det{g_{\mu\nu}}}dx^1\cdots dx^n$ is the Riemannian volume form and $\tr$ denotes the matrix trace on the fibres of $\End(V)$.
The endomorphism $a_n(x,P)$ can be uniquely determined by localizing the heat trace by an auxiliary  smooth endomorphism $F$ of $V$, called a smearing endomorphism. 
The localized heat trace $\Tr(Fe^{-tP})$ has an asymptotic expansion as $t\to 0^+$ of the form \cite[Chapter 3]{Gilkey2004}
\begin{eqnarray}\label{smearedheatasymp}
\Tr(Fe^{-tP})\sim \sum_{n=0}^\infty a_n(F,P)t^{\frac{n-m}2},
\end{eqnarray}
with 
\begin{equation}\label{localizedcoef}
a_n(F,P)=\int_M \tr\big(F(x)a_n(x,P)\big) dx.
\end{equation}

If $P$ is a Laplace type operator i.e., its leading symbol is given by the metric tensor,  then the densities $a_n(x,P)$ can be expressed in terms of the Riemannian curvature, an endomorphism $E$, and their derivatives. 
The endomorphism $E$ measures how far the operator $P$ is from being the Laplacian $\conn^*\conn$ of a connection $\conn$ on $V$;   
\begin{equation}\label{decomtolapandend}
E=\conn^*\conn-P.
\end{equation}
Such a connection and endomorphism are unique for the given Laplace type operator $P$ \cite[Lemma 4.1.1]{Gilkey1995} 
The first two densities of the heat equation for such $P$ are given by \cite[Theorem 3.3.1]{Gilkey2004} 
\begin{eqnarray}
a_0(x,P)&=&(4\pi)^{-m/2}{\rm I},\label{a0}\\
a_2(x,P)&=&(4\pi)^{-m/2}\big(\frac{1}{6}\scalar(x)+E\big).
\end{eqnarray}
Here ${\rm I}$ denotes the identity bundle map on $V$.
For the Laplacian on functions $\lap_0$, the connection is the de Rham exterior derivative $d:C^\infty(M)\to \Omega^1(M)$, and obviously $E=0$. 
Hence the two first terms in the heat kernel of $\lap_0$ are given by 
\begin{eqnarray}
a_0(x,\lap_0)&=&(4\pi)^{-m/2},\label{a0functions}\\
a_2(x,\lap_0)&=&(4\pi)^{-m/2}\,\frac{1}{6}{\scalar(x)}.\label{a2functions}
\end{eqnarray}
In the case of the Laplacian on one forms $\lap_1:\Omega^1(M)\to \Omega^1(M)$, also called the Hodge-de Rham Laplacian, the connection in \eqref{decomtolapandend} is the Levi-Civita connection on the cotangent bundle. The endomorphism $E$ is the 
 negative of the Ricci operator, $E=-\ric$, on the cotangent bundle,
   which is defined by raising the first index of the Ricci tensor (denoted by $\ric$ as well); 
\begin{equation*}
\ric_x(\alpha^\sharp,X)=\ric_x(\alpha)(X),\quad \alpha\in T_x^*M,\,  X\in T_xM.
\end{equation*}
Therefore, one has
\begin{eqnarray}
a_0(x,\lap_1)&=&(4\pi)^{-m/2}{\rm I},\\ \label{a01forms}
a_2(x,\lap_1)&=&(4\pi)^{-m/2}\big(\frac{1}{6}{\scalar(x)}-\ric_x\big). \label{a21forms}
\end{eqnarray}
Furthermore, the function  $\tr(F)$  is smooth for every $F\in C^\infty(\End(T^*M))$, and can be used as a smearing function to localize the heat trace of $\lap_0$. 
Then \eqref{a2functions} and \eqref{a21forms} lead to the identity 
\begin{equation}\label{differenceofterms}
a_2(\tr(F),\lap_0)-a_2(F,\lap_1)=(4\pi)^{-m/2}\int_M \tr\big(F(x)\ric_x\big)dx ,
\end{equation}
for every $F\in C^\infty(\End(T^*M)),$ which motivates the following definition. 
\begin{definition}\label{riccifunctional}
The { Ricci functional} is the functional 
on $C^\infty(\End(T^*M))$ defined as 
\begin{equation*}
\ricfun(F)=a_2(\tr(F),\lap_0)-a_2(F,\lap_1),\quad F\in C^\infty(\End(T^*M)).
\end{equation*}
\end{definition} 
\begin{proposition} For an oriented closed Riemannian manifold $M$ of dimension $m$, we have 
\begin{equation*}
\Tr\left(\tr(F)e^{-t\lap_0}\right)-\Tr\left(F e^{-t\lap_1}\right)\sim \ricfun(F)\, t^{1-\frac{m}{2}} \text{ as } t\to 0^+.  
\end{equation*}
\end{proposition}
\begin{proof}
By \eqref{a0} and \eqref{localizedcoef}, we have $\tr(F)a_0(x,\lap_0)=\tr(F(x)a_0(x,\lap_1)).$
This implies that
\begin{equation}\label{equalityoffirstterms}
a_0(\tr(F),\lap_0)=a_0(F,\lap_1),\qquad F\in C^\infty(\End(T^*M)).
\end{equation} 
The   asymptotic expansion of the localized heat kernel \eqref{smearedheatasymp} then shows that the first terms will cancel each other. The  difference of the second terms, which are multiples of $t^{1-\frac{m}{2}}$, will become the first term in the asymptotic expansion of the differences of localized heat kernels.  
\end{proof}

\subsection{Spectral zeta function and the Ricci functional} \label{zetaformulationofricci} 

The other spectral function assigned to a positive elliptic operator $P$ is the 
{ spectral zeta function} defined by 
\begin{equation*}
\zeta(s,P)=\Tr(P^{-s}({\rm I}-\pr)),\quad \Re (s)\gg 0,
\end{equation*} 
where $\pr $ is the projection on the kernel of $P$, which is finite dimensional. 
The localized version of the spectral zeta function is $\zeta(s,F,P)=\Tr(FP^{-s}({\rm I}-\pr))$. 
The spectral zeta function has a meromorphic extension to the complex plane $\mathbb{C}$ with  isolated  simple poles \cite[Lemma 1.3.7]{Gilkey2004}.
The residue at the poles, and also the values of the function at some of the negative real numbers, are related to the coefficients  of heat kernel (cf. 
\cite[Lemma 1.12.2]{Gilkey1995}).
This enables us to write the Ricci functional in terms of zeta functions.
\begin{proposition}
For an orientable closed Riemannian manifold $M$ of dimension $m>2$, we have 
\begin{equation}\label{zetariccimnot2}
\ricfun(F)=
 \Gamma(\frac{m}{2}-1)\res_{s=\frac{m}{2}-1}\Big(\zeta(s,\tr(F),\lap_0)-\zeta(s,F,\lap_1)\Big),\quad F\in C^\infty(\End(T^*M)).
\end{equation}
Moreover, if $M$ is two dimensional manifold of genus $g$, then we have 
\begin{equation}\label{zetariccim2}
\ricfun(F)=\zeta(0,\tr(F),\lap_0)-\zeta(0,F,\lap_1)+\Tr(\tr(F)\pr_0)-\Tr(F\pr_1),
\end{equation}
where $\pr_j$ is the projection on the kernel of Laplacian $\lap_j$, $j=0,1$.
\end{proposition}
\begin{proof}
The relation between the values and residues of the localized zeta function of a second order positive elliptic operator $P$ and its localized heat trace coefficients \eqref{localizedcoef}  is  given by (cf. \cite[Lemma 1.3.7]{Gilkey2004})
\begin{equation}\label{residuezetaan}
a'_n(F,P)=\res_{s=\frac{m-n}{2}}\Big(\Gamma(s)\zeta(s,F,P)\Big),
\end{equation}
where 
$$a'_n(F,P)=
\begin{cases}
 a_m(F,P)-\Tr(FQ) & \text{ if } n=m\\ 
a_n(F,P) & \text{ if }n\neq m.
\end{cases}$$
If $m>2$, then 
\begin{equation*}
a_2(\tr(F),\lap_0)=\res_{\frac{m}{2}-1}\Gamma(s)\zeta(s,\tr(F),\lap_0).
\end{equation*}
Since the gamma function is regular at $s=m/2-1$, the right hand side is equal to $\Gamma(m/2-1)\res_{s=(m/2)-1} \zeta(s,\tr(F),\lap_0)$.
Similar argument applies to $\lap_1$ 
\begin{equation*}
a_2(F,\lap_1)=\Gamma(m/2-1)\res_{s=m/2-1}\zeta(s,F,\lap_1).
\end{equation*}
Combining them, we obtain \eqref{zetariccimnot2}.
In dimension two, on the other hand, we have 
\begin{equation*}
\ricfun(F)-\Tr\big(\tr(F)\pr_0\big)+\Tr(FQ_1)=\res_{s=0}(\Gamma(s)(\zeta(s,\tr(F),\lap_0)-\zeta(s,F,\lap_1)).
\end{equation*}
The zeta function of $\lap_j$, $j=0,1$, is regular at zero  hence the right hand side is equal to 
$$\Big(\zeta(0,\tr(F),\lap_0)-\zeta(0,F,\lap_1)\Big)\res_{s=0}\Gamma(s)=\zeta(0,\tr(F),\lap_0)-\zeta(0,F,\lap_1).$$
Putting it all together, \eqref{zetariccim2} is proven.
\end{proof}
\begin{remark}
By combining \eqref{residuezetaan} and \eqref{equalityoffirstterms}, it can be shown that
the difference of zeta functions $\zeta(s,\tr(F),\lap_0)-\zeta(s,F,\lap_1)$ is regular at $m/2$, and its first pole is located at $s=m/2-1$.
\end{remark}

\subsection{Ricci functional and the de Rham spectral triple}\label{RiccifundeRahm}

At first glance, it seems that the ingredients  used  to define the Ricci functional 
do not come from a spectral triple.
In other words, $\lap_0$ and $\lap_1$ are not the Laplacian of a Dirac type operator.
Nevertheless, they are part of the Laplacian of the de Rham complex, which as an elliptic complex gives rise to a Dirac type operator on forms.
By restricting ourselves to smooth endomorphisms of the cotangent bundle, and producing the right smearing endomorphism, we obtain the formalism given in Definition \ref{riccifunctional}.
In this section, we will investigate how the Ricci functional can be defined using the de Rham complex spectral triple.

Let $M$ be a closed orientable Riemannian manifold. 
Consider the de Rham spectral triple,
$$(C^\infty(M),L^2(\Omega^{ev}(M)) \oplus L^2(\Omega^{odd} (M)),d+\delta, \gamma),$$ which is the even spectral triple constructed from the de Rham complex.
Here, $d$ and $\delta$  denote the exterior derivative and  coderivative on the exterior algebra, and $\gamma$ is the $\mathbb{Z}_2$-grading on the forms whose eigenspace for eigenvalues 1 and -1 are even and odd forms, respectively.
 The full Laplacian on forms $\lap=d\delta+\delta d$ is the Laplacian of the Dirac operator $d+\delta$, and is the direct sum of Laplacians on $p$-forms, $\lap=\oplus \lap_p$. As a Laplace type operator, $\lap$ can be written as $\conn^*\conn-E$ by Weitzenb\"ock formula, where $\conn$ is the Levi-Civita connection extended to all forms and 
\begin{equation*}
E=-\frac12c(dx^\mu)c(dx^\nu)\Omega(\partial_\mu,\partial_\nu).
\end{equation*} 
Here $c$ denotes the Clifford multiplication and $\Omega$ is the curvature operator of the Levi-Civita  connection acting on exterior algebra 
(cf. \cite[Lemma 3.2.1]{Gilkey1995}). 
The restriction of $E$ to one forms gives the Ricci operator. 

To work with the Laplacian on one forms, we will use smooth endomorphisms $F$ of the cotangent bundle.   
The smearing endomorphism $\tilde{F}=\tr(F){\rm I}_0 \oplus F\in C^\infty(\End(\bigwedge^\bullet M))$, where ${\rm I}_0$ denotes the identity map on functions, can be used to localize the heat kernel of the full Laplacian and
\begin{equation}\label{riccifunctionalgamma}
\ricfun(F)=a_2(\gamma \tilde{F},\lap),\quad F\in C^\infty(\End(T^*M)).
\end{equation}
With the above notation, we can also rewrite the formulae \eqref{zetariccimnot2}, \eqref{zetariccim2} as 
\begin{equation}\label{Riccizetagamma}
\ricfun(F)=\begin{cases}
 \Gamma(\frac{m}{2}-1)\res_{s=\frac{m}{2}-1}\zeta(s,\tilde{F}\gamma,\lap)& m>2,\\&\\
\zeta(0,\gamma\tilde{F},\lap)+\Tr(\tr(F)\pr_0)-\Tr(F\pr_1)& m=2.
\end{cases}
\end{equation}

\section{Ricci functional for the noncommutative two tori}
We start this section by briefly reviewing several classical facts regarding the noncommutative two torus.
Then  we construct the modular de Rham spectral triple for the noncommutative two torus  for the conformally flat metrics.
The Laplacian of this spectral triple will be used to define the Ricci functional for the noncommutative two torus, and  then the Ricci density  will  explicitly be computed. 
\subsection{The noncommutative two torus $\A$}
Let $\theta$ be a real number. The noncommutative two torus $\A$ is the universal $C^\ast$-algebra generated by two unitary elements $U$ and $V$, $U^*=U^{-1}$ and $V^*=V^{-1}$, which satisfy the commutation relation
\begin{equation*}
VU=e^{2\pi i\theta}UV.
\end{equation*} The $C^*$-algebra $\A$ is naturally acted upon by $\mathbb{R}^2$, the action being given by the two-parameter group $\{\alpha_s\}_{s\in\mathbb{R}^2}$ of *-automorphisms 
\begin{equation}\label{actionofR2onA}
\alpha_s(U^nV^m)=e^{i(s_1n+s_2m)}U^nV^m,\quad s=(s_1,s_2)\in \mathbb{R}^2.
\end{equation}
The set of all elements $a\in\A$ for which the map $\mathbb{R}^2\ni s\mapsto \alpha_s(a)$ is smooth is an involutive subalgebra of $\A$, which will henceforth be denoted as $\Ai$. The algebra $\Ai$ admits the alternative description 
$$\Ai=\left\{\sum_{(m.n)\in \mathbb{Z}^2} a_{mn}U^mV^n| \{a_{mn}\}_{(m,n)\in \mathbb{Z}^2} \text{ is rapidly decaying}\right\}.$$
The infinitesimal generators of the action $\alpha$ on $\A$ in the direction of $e_1=(1,0)$ and $e_2=(0,1)$ respectively are the derivations $\delta_1$ and $\delta_2$ given by
\begin{eqnarray*}
\delta_1(U)=U,\quad && \delta_1(V)=0\\
\delta_2(U)=0,\quad && \delta_2(V)=V.
\end{eqnarray*} We note that $\delta_j(a)^*=-\delta_j(a^*)$, for $a\in \Ai$.

If $\theta$ is an irrational number, then it is well known that $\A$ is simple  and has a unique  tracial state $\varphi$, which acts on $\Ai$ as
\begin{equation*}
\varphi(\sum a_{mn}U^mV^n)=a_{00}.
\end{equation*}
The trace $\varphi$ is $\alpha$-invariant. Consequently $\varphi\circ \delta_j=0$ and $\varphi(a\delta_j(b))=-\varphi(\delta_j(a)b)$, for all $a,\,b\in \Ai$ and $j=1,\,2$.
The Hilbert space obtained by completing $\A$ with respect to the norm associated to the following inner product, is denoted by $\mathcal{H}_0$; 
\begin{equation*}
\langle a,b\rangle=\varphi(b^*a),\quad a,b\in\A.
\end{equation*}
For $\theta=0$, $A_\theta$ is the algebra $C(\mathbb{T}^2)$ of continuous functions on the torus $\mathbb{T}^2$ and $\delta_1$ and $\delta_2$ turn to $\frac{1}{i}\frac{\partial}{\partial x}$ and $\frac{1}{i}\frac{\partial}{\partial y}$  and $\varphi$ is nothing but integration with the volume form $dxdy$.

\subsection{The de Rham spectral triple for the noncommutative two torus}
 In this section, we describe the de Rham spectral triple of the noncommutative two torus $\A$. For this purpose, consider the vector space $W=\mathbb{R}^2$, and let $\tau$ be a complex number in the upper half plane, i.e. $\Im(\tau)>0$. Let  $g_\tau$ be the positive definite symmetric bilinear form on $W$  given by 
\begin{equation}\label{flattaumetricontangent}
g_\tau=\frac{1}{\Im(\tau)^2}
\begin{pmatrix}
|\tau|^2 & -\Re(\tau)\\ 
-\Re(\tau) & 1 
\end{pmatrix}.
\end{equation} 
The inverse $g_\tau^{-1}=\begin{pmatrix}
1 & \Re(\tau)\\ 
\Re(\tau) & |\tau|^2 
\end{pmatrix}$ of $g_\tau$ is a metric on the dual of $W$. 
The entries of $g_\tau^{-1}$ will be denoted by $g^{jk}$, $1\leq j,\,k\leq 2$.

Let now $\extp^\bullet (W^*_\mathbb{C},g^{-1}_\tau)$ be the  exterior algebra of $W^*_\mathbb{C}=(W\otimes \mathbb{C})^*$.
The algebra $$\Ai\otimes \extp^\bullet (W^*_\mathbb{C},g^{-1}_\tau),$$
will then play the role of the algebra of differential forms of the noncommutative two torus $\A$. 
In this framework, the Hilbert space of functions, denoted $\mathcal{H}^{(0)}$, is simply the Hilbert space given by the GNS construction of $\Ai$ with respect to 
$\frac{1}{\Im(\tau)}\varphi$. Additionally, the
Hilbert space of one forms, denoted $\mathcal{H}^{(1)}$, is the space $\mathcal{H}_0\otimes (\mathbb{C}^2,g_\tau^{-1})$ with inner product given by  
\begin{equation}\label{oneforminnerproduct}
\langle a_1\oplus a_2, b_1\oplus b_2  \rangle=  \frac{1}{\Im(\tau)}\sum_{j,k} g^{jk}\varphi(b_k^*a_j),\;\;a_i,\,b_i\in \Ai,
\end{equation}
while the Hilbert space of two forms, denoted $\mathcal{H}^{(2)}$, is given by the GNS construction of $\Ai$ with respect to $\Im(\tau)\, \varphi$.

The exterior derivative on elements of $\Ai$ is given by
\begin{equation}\label{donfunctionstau}
a\mapsto i\delta_1(a)\oplus i\delta_2(a),\quad a\in \Ai.
\end{equation}
It will be denoted by $d_0$, when considered as a densely defined operator from $\mathcal{H}^{(0)}$ to $\mathcal{H}^{(1)}$.
The operator $d_1: \mathcal{H}^{(1)}\to \mathcal{H}^{(2)}$ is defined on the elements of $\Ai\oplus\Ai$ as 
\begin{equation}\label{donformstau}
a\oplus b\mapsto i\delta_1(b)-i\delta_2(a),\qquad a,b\in\Ai.
\end{equation}
The adjoints of the operators $d_0: \mathcal{H}^{(0)}\to \mathcal{H}^{(1)}$ and $d_1:\mathcal{H}^{(1)}\to \mathcal{H}^{(2)}$ are then given by 
\begin{eqnarray*}
d_0^*(a\oplus b) &=& -i \delta_1(a) -i \Re(\tau) \delta_2(a)-i\Re(\tau) \delta_1(b)-i|\tau|^2 \delta_2(b),\\
d_1^*(a) &=& (i|\tau|^2 \delta_2(a)+i\Re(\tau) \delta_1(a))\oplus (-i\Re(\tau) \delta_2(a)-i\delta_1(a)),
\end{eqnarray*} for all $a,b\in \Ai$.
\begin{definition}\label{deRhamspectralflat}
 The (flat) de Rham spectral triple of $\A$ is the even spectral triple 
$(\A,\mathcal{H},D),$
where $\mathcal{H}=\mathcal{H}^{(0)}\oplus\mathcal{H}^{(2)} \oplus \mathcal{H}^{(1)}$, $ D=\begin{pmatrix}
0& d^*\\ d& 0
\end{pmatrix},$ and $d=d_0+d_1^*$.
\end{definition}
\noindent Note that the operator $d$ and its adjoint $d^*=d_1+d_0^*$ act on $\Ai\oplus \Ai$ as
\begin{equation}\label{nonpertddstar}
d=\begin{pmatrix}
i\delta_1 & i|\tau|^2 \delta_2+i\Re(\tau) \delta_1\\
i\delta_2 & -i\Re(\tau) \delta_2-i\delta_1
\end{pmatrix},\quad
d^*=\begin{pmatrix}
-i\delta_1-i\Re(\tau) \delta_2 &-i\Re(\tau) \delta_1- i|\tau|^2 \delta_2\\
-i\delta_2 & i\delta_1
\end{pmatrix}. 
\end{equation}
Note also that the de Rham spectral triple introduced in Definition \ref{deRhamspectralflat}, is isospectral to the de Rham complex spectral triple of the flat torus $\mathbb{T}^2$ with the metric given by \eqref{flattaumetricontangent}.

\subsection{The modular de Rham spectral triple }\label{modular}

In this section, we introduce the modular de Rham spectral triple for the noncommutative two torus $\A$, which resembles the de Rham complex on the torus when its flat metric is conformally perturbed. 
We shall show that the modular de Rham spectral triple is the transposed spectral triple, in the sense of \cite{Connes-Moscovici2014}, of a modular spectral triple obtained by perturbing certain spectral triple by a Weyl factor $k=e^{h/2}$. 

The impact of a conformal change of the Riemannian metric on the de Rham complex is prescribed by the change of pointwise inner products on forms and by the change of the Riemannian volume form. 
Nonetheless, it can be represented by simply multiplying the volume form by an appropriate function.  
For instance, on a manifold of dimension $m$, if we denote $e^{-h}g$ by $\tilde{g}$, then the new volume form $\tilde{dx}$ will be $e^{-\frac{m}{2}h}dx$,  $\tilde{g}^{-1}=e^hg^{-1}$ on $T^*M$, and $\wedge^2\tilde{g}^{-1}=e^{2h}\wedge^2 g^{-1}$ on $\wedge^2 T^*M$. 
As a result, the inner products on functions and on two forms will change as follows:
\begin{align*}
\langle f_1,f_2\rangle_{\tilde{g}}&=\int_M f_1 {f}_2 e^{-\frac{m}{2}h}dx \quad &f_1,f_2\in \Omega^0(M),\\
\langle \alpha_1,\alpha_2\rangle_{\tilde{g}}&=\int_M g^{-1}(\alpha_1,\alpha_2) e^{\frac{2-m}{2}h}dx \quad &\alpha_1,\alpha_2\in \Omega^1(M),\\
\langle \omega_1,\omega_2\rangle_{\tilde{g}}&=\int_M \wedge^2g^{-1}(\omega_1,{\omega}_2) e^{\frac{4-m}{2}h}dx\quad &\omega_1,\omega_2 \in \Omega^2(M).
\end{align*}  
Note that in the two-dimension case, the inner product on one forms remains unchanged.

The above classical fact enables us to study the conformal change of metric for the noncommutative two torus, by perturbing the tracial state $\varphi$ by a noncommutative Weyl factor.
This strategy was considered in \cite{Connes-Cohen1992,Connes-Tretkoff2011} with the purpose of analyzing the conformal perturbation of metric for the Dolbeault complex on the noncommutative two torus.
In this paper, we will employ the same strategy to find the de Rham complex and hence de Rham spectral triple for the noncommutative two torus when the metric is perturbed conformally.

The conformal perturbation of the metric on the noncommutative two torus is implemented by changing the tracial state $\varphi$ by a noncommutative Weyl factor $e^{-h}$, where the dilaton $h$ is a selfadjoint smooth element of the noncommutative two torus, $h=h^*\in\Ai$.  
The conformal change of the metric by the Weyl factor $e^{-h}$ will change the inner product on functions and on two forms as follows.
On functions, the Hilbert space given by GNS construction of $\A$ with respect to the positive linear functional $\varphi_{0}(a)=\frac{1}{\Im(\tau)}\varphi(ae^{-h})$ will be denoted by $\mathcal{H}_h^{(0)}$. 
Therefore the inner product of $\mathcal{H}_h^{(0)}$ is given by   
\begin{equation*}
\langle a,b \rangle_0=\frac{1}{\Im(\tau)}\varphi(b^*ae^{-h}),\quad a,b\in \A.
\end{equation*}
On one forms, the Hilbert space will stay the same as in \eqref{oneforminnerproduct}, and will be denote by $\mathcal{H}^{(1)}_h$.
 On the other hand, the Hilbert space of two forms, denoted by $\mathcal{H}^{(2)}_h$, is the Hilbert space given by the GNS construction of $\A$ with respect to $\varphi_2(a)=\Im(\tau)\varphi(a e^{h})$. 
 Hence its inner product is given by
\begin{equation*}
\langle a,b \rangle_2={\Im(\tau)}\varphi(b^*ae^{h}),\quad a,b\in \A.
\end{equation*}
The  positive functional $a\mapsto \varphi(ae^{-h})$, called the conformal weight, is a twisted trace of which modular operator is given by  
$$\Delta(a)=e^{-h}ae^{h},\quad a \in\A.$$
The logarithm $\log\Delta$ of the modular operator will be denoted by $\nabla$, and is given by $\nabla(a)=-[h,a]$ (for more details, see \cite{Connes-Cohen1992,Connes-Tretkoff2011}).

The exterior derivatives are defined in the same way they were defined in the flat case \eqref{donfunctionstau} and \eqref{donformstau}. However, to emphasize that they are acting on different Hilbert spaces, we will denote them by $d_{h,0}:\mathcal{H}_h^{(0)}\to \mathcal{H}_h^{(1)}$ and $d_{h,1}:\mathcal{H}_h^{(1)}\to \mathcal{H}_h^{(2)}$. 

\begin{lemma}\label{adjointlemma}
Let $k=e^{h/2}$ where $h=h^*\in\Ai$ is the dilaton of the conformal weight. Then the adjoint operators of $d_{h,0}$ and $d_{h,1}$ are given by 
\begin{eqnarray*}
d^*_{h,0}&=&  R_{k^2}\comp  d_0^*,\\
d^*_{h,1} &=&  d_1^* \comp R_{k^2}\,,
\end{eqnarray*} 
where $R_{k^2}$ denotes the right multiplication by $k^2$.
\end{lemma}
\begin{proof}
It is enough to show that the adjoint of $\delta_j:\mathcal{H}_n\to \mathcal{H}_{n'}$, $j=1,2$, is equal to 
$$R_{k^{-2n'}}\comp\delta_j\comp R_{k^{2n}}.$$ 
If for any $n\in\mathbb{Z}$,  the Hilbert space given by the GNS constructions of $\A$ with respect to the positive functional $a\mapsto \varphi(ae^{nh})$ is denoted by $\mathcal{H}_n$. 
For this purpose, let $a,\,b\in \Ai$. Then 
\begin{align*}
\langle \delta_j(a),b \rangle_n &=\varphi(b^*\delta_j(a)e^{nh}) \\
&=\varphi((be^{nh})^*\delta_j(a))\\
&=\varphi(\delta_j(be^{nh})^*a)\\
&= \langle a,\delta_j(b e^{nh})e^{-n'h}\rangle_{n'},
\end{align*} 
which concludes the proof.\end{proof}

Next, we consider the Hilbert spaces $\mathcal{H}^{+}_h=\mathcal{H}_h^{(0)}\oplus  \mathcal{H}_h^{(2)}$ and $\mathcal{H}^{-}_h=\mathcal{H}_h^{(1)}$, and the operator $d_h:\mathcal{H}^{+}_h\to \mathcal{H}^{-}_h$, $d_h=d_{h,0}+ d_{h,1}^*$. Therefore
\begin{equation*}
d_h=\begin{pmatrix}
i\delta_1 & \big(i|\tau|^2 \delta_2+i\Re(\tau) \delta_1\big)\comp R_{k^2}\\
i\delta_2 & \big(-i\Re(\tau) \delta_2-i\delta_1\big)\comp R_{k^2}
\end{pmatrix},
\end{equation*}
and its adjoint is given by
 \begin{equation*}
d_h^*=\begin{pmatrix}
R_{k^2}\comp \big( i\delta_1-i\Re(\tau) \delta_2\big) &R_{k^2}\comp\big(-i\Re(\tau) \delta_1- i|\tau|^2 \delta_2\big)\\
-i\delta_2 & i\delta_1
\end{pmatrix} .
\end{equation*}
We also consider the operator   
\begin{equation*}
D_h=\begin{pmatrix}
0 & d_h^*\\ d_h & 0
\end{pmatrix},
\end{equation*}
which acts on $\mathcal{H}_h= \mathcal{H}^{+}_h\oplus\mathcal{H}_h^{-}$. Define the Hilbert space $\mathcal{H}=\mathcal{H}_0\oplus\mathcal{H}_0\oplus\mathcal{H}_0\oplus\mathcal{H}_0$ and the unitary operator $W:\mathcal{H}\to \mathcal{H}_h,$ $$W=R_k\oplus R_{k^{-1}}\oplus {\rm {I}\sb{\mathcal{H}_0\oplus\mathcal{H}_0}}.$$ The operator $D_h$ can be transferred to an operator $\tilde{D}_h$ on $\mathcal{H}$ by the inner perturbation 
$$\tilde{D}_h:=W^* D_h W=\begin{pmatrix} 0 &  R_k \comp d^*\\ d\comp R_k & 0 \end{pmatrix}.$$

In order to define the modular de Rham spectral triple for the noncommutative two torus, we employ the following constructions from \cite{Connes-Moscovici2014}. 
Let $(\mathcal{A},\mathcal{H}^+\oplus \mathcal{H}^-,D)$ be an even spectral triple with  grading operator $\gamma$, where $D=\begin{pmatrix}
0 & T^*\\ T & 0
\end{pmatrix}$  and $T:\mathcal{H}^+\to\mathcal{H}^-$ is an unbounded operator with adjoint $T^*$. If $f\in\mathcal{A}$ is positive and invertible,
then $(\mathcal{A,H},D_{(f,\gamma)})$ is a modular spectral triple with respect to the inner automorphism $\sigma(a)=faf^{-1}$, $a\in \mathcal{A}$  ( \cite[Lemma 1.1]{Connes-Moscovici2014}), where the Dirac operator is given by $$D_{(f,\gamma)}=\begin{pmatrix}
0 & f T^*\\ Tf & 0
\end{pmatrix}. $$
On the other hand, any modular spectral triple $(\mathcal{A,H},D)$ with an automorphism $\sigma$ admits a transposed modular spectral triple $(\mathcal{A}^{\rm op},\bar{\mathcal{H}}, D^t) $ \cite[Proposition 1.3]{Connes-Moscovici2014}, where $\mathcal{A}^{\rm op}$ is the opposite algebra of $\mathcal{A}$, $\bar{\mathcal{H}}$ is the dual Hilbert space, the action of $\mathcal{A}^{\rm op}$ on $\bar{\mathcal{H}}$ is the transpose of the the action of $\mathcal{A}$ on $\mathcal{H}$, $D^t$ is the transpose of $D$, and $\sigma'$ is the automorphism of $\mathcal{A}^{\rm op}$ given by $\sigma'(a^{\rm op})=(\sigma^{-1}(a))^{\rm op}$.

\begin{proposition}
Let $k=e^{h/2}$, where $h=h^*\in\Ai$.
 The triple $(\A^{\rm op},\mathcal{H},\tilde{D}_h)$  is a modular spectral triple, where the automorphism  of $\A^{\rm op}$ is given by
\begin{equation*}
a^{\rm op}\mapsto (k^{-1}ak)^{\rm op},\quad a\in \Ai,
\end{equation*} 
and the representation of $\A^{\rm op}$ on $\mathcal{H}$ is given by the right multiplication of $\A$ on $\mathcal{H}$.
Moreover, the transposed of the modular spectral triple $(\A^{\rm op},\mathcal{H},\tilde{D}_h)$ is isomorphic to the perturbed spectral triple 
\begin{equation}\label{finalspectraltriple}
(\A,\mathcal{H},\bar{D}_h),\quad 
\bar{D}_h=\begin{pmatrix}
0 & k d\\ d{}^* k & 0
\end{pmatrix},
\end{equation}
where the operators $d$ and $d^*$ are as in \eqref{nonpertddstar}.
\end{proposition}

\begin{proof}
First of all, we note that $\bar{D}_h=(D^t)_{(k,\gamma)}$ where $D$ is the Dirac operator of the (flat) de Rham spectral triple (Definition \ref{deRhamspectralflat}) and $D^t=\begin{pmatrix}
0 & d\\ d{}^*  & 0
\end{pmatrix}$ is the transposed of $D$. 
Hence $(\A,\mathcal{H},\bar{D}_h)$ is a modular spectral triple with automorphism $\sigma(a)=kak^{-1}$.

We remark that the adjoint map $a\mapsto a^*$ defines an anti-unitary operator $J:\mathcal{H}\to\mathcal{H}$, or equivalently a unitary map from $\mathcal{H}$ to its complex conjugate $\mathcal{H}^c$.
We transform the spectral triple \eqref{finalspectraltriple} using $J$.
The action of $\operatorname{Ad}J$ on $A_\theta$ will transform an element $a\in \A$ into $JaJ^{-1}=R_{a^*}$, and 
$$J\bar{D}_hJ^{-1}= \begin{pmatrix}
0  & R_k\comp d^*\\ d\comp R_k  & 0
\end{pmatrix}.$$
Let $J_\mathcal{H}:\mathcal{H}^c\to \bar{\mathcal{H}}$ be the unitary operator $\mathcal{H}\ni\xi\mapsto J_\mathcal{H}(\xi)=\langle \cdot, \xi\rangle\in \bar{H}$.
One can readily see that $T^t=J_\mathcal{H} T^* J^{-1}_\mathcal{H}$, for every $T\in B(\mathcal{H})$.
Then for any element $a\in \Ai$, acting as an operator on $\mathcal{H}$, we have, 
$$
J_\mathcal{H} J a J^{-1} J^{-1}_\mathcal{H}=J_\mathcal{H}R_{a^*} J^{-1}_\mathcal{H}=J_\mathcal{H}R_{a}^* J^{-1}_\mathcal{H}=R_a{}^t=a.$$ 
On the other hand, 
$$J_\mathcal{H} J\bar{D}_h J^{-1}J^{-1}_\mathcal{H}=J^{-1}_\mathcal{H}\begin{pmatrix}
0  & R_k\comp d^*\\ d\comp R_k  & 0
\end{pmatrix}J^{-1}_\mathcal{H}=\tilde{D}_h{}^t.
$$
Therefore the spectral triple $(\A,\mathcal{H},\bar{D}_h)$ is transformed, under the action of the unitary map $J_\mathcal{H} J$, into the transposed spectral triple $(\A,\bar{\mathcal{H}},\tilde{D}_h^t)$ of $(\A^{\rm op},\mathcal{H},\tilde{D}_h)$. With respect to this identification, the automorphism $\sigma$ remains unchanged. Consequently the inner automorphism for $(\A,\bar{\mathcal{H}},\tilde{D}_h^t)$ is given by
$\sigma'(a^{\rm op})=(\sigma^{-1}(a))^{\rm op}=(k^{-1}ak)^{\rm op}.$
\end{proof}

\begin{definition}\label{modularspectraltriple}
The modular spectral triple $(\A,\mathcal{H},\bar{D}_h)$ in \eqref{finalspectraltriple} will be called the {modular de Rham spectral triple} of the noncommutative two torus with dilaton $h$.
\end{definition}

\subsection{Laplacian of the modular de Rham spectral triple}

The geometric information of the modular de Rham spectral triple can be extracted from the spectrum of its Laplacian $\bar{D}_h^2$.     
A straightforward computation shows that the Laplacian is the direct sum of three components, denoted by $\lap_{h,0}, \lap_{h,1}$ and $\lap_{h,2}$, which are the analogues of the Laplacian on functions, on one forms, and on two forms respectively, i.e.,
\begin{equation}\label{Lapisdirectsumtau}
\bar{D}_h{}^2=\begin{pmatrix}
kdd^*k &0\\
0 & d^*k^2d
\end{pmatrix}=\lap_{h,0}\oplus \lap_{h,2}\oplus \lap_{h,1}.
\end{equation}
If we denote the flat Laplacian on functions by 
$$\lap_0=d_0^*d_0=\delta_1^2+2\Re(\tau)\delta_1\delta_2+|\tau|^2\delta_2^2,$$
then
\begin{eqnarray*}
\lap_{h,0}=\lap_{h,2}&=&k\lap_0k.
\end{eqnarray*}
Moreover, the operator $\lap_{h,1}:\mathcal{H}^{(1)}\to \mathcal{H}^{(1)}$ is given on $\Ai\oplus \Ai$ by  
\begin{eqnarray}\label{lap1}
\begin{split}
\lap_{h,1}&=
\big(\delta_1k^2\delta_1+\Re(\tau)\delta_1 k^2 \delta_2+\Re(\tau)\delta_2 k^2 \delta_1+|\tau|^2\delta_2k^2\delta_2\Big)\otimes {\rm I_2}\\
& \qquad\qquad\qquad+\big( \delta_1k^2\delta_2-\delta_2k^2\delta_1\big)\begin{pmatrix}
0 & \Im(\tau)^2  \\
 -1& 0
\end{pmatrix}.
\end{split}
\end{eqnarray}

The Laplacian $\bar{D}_h^2$ is directly related to the Laplacian $\bar{D}^2_\varphi$ of the modular spectral triple  $(\Ai,\mathcal{H}, \bar{D}_\varphi)$ of weight $\varphi$ defined in \cite[Definition 1.10]{Connes-Moscovici2014}, as follows. By \cite[Lemma 1.11]{Connes-Moscovici2014}, the Laplacian $\bar{D}_\varphi^2$ is the direct sum of two parts 
 $$\bar D^2_\varphi=\begin{pmatrix}
 \lap_\varphi & 0\\ 0& \lap_\varphi^{(0,1)}
 \end{pmatrix},$$
 where $\lap_\varphi$ is equal to our $\lap_{h,0}$ and  
\begin{equation*}
\lap_\varphi^{(0,1)}=(\delta_1+\tau \delta_2)k^2(\delta_1+\bar\tau\delta_2)=\delta_1k^2\delta_1+\bar\tau \delta_1k^2\delta_2+\tau\delta_2k^2\delta_1+|\tau|^2\delta_2k^2\delta_2.
\end{equation*}
Furthermore, one can readily see that $\lap_{h,1}$ is a first order perturbation of $\lap_\varphi^{(0,1)}$, namely 
\begin{equation}\label{dallap-deRhamlap}
\lap_{h,1}=\lap_\varphi^{(0,1)}\otimes {\rm I}_2+\Big(\delta_1(k^2)\delta_2-\delta_2(k^2)\delta_1\Big)\otimes\sigma,
\end{equation}
where $\sigma=\begin{pmatrix}
i\Im(\tau) & \Im(\tau)^2  \\
 -1& i\Im(\tau)
\end{pmatrix}.$

\subsection{Ricci functional and Ricci density for the modular de Rham spectral triple}

Using the pseudodifferential calculus with  $\Ai\otimes M_4(\mathbb{C})$-valued symbols (see Section \ref{computation}), one can show that the localized heat trace of  $\bar{D}_h^2$  has an asymptotic expansion of the form \eqref{smearedheatasymp}, and its coefficients are of the following form (see the \S\ref{Heatkernelpseudo});  
\begin{equation*}
a_n(E,\bar D_h^2)=\varphi\big(\tr\left(E\, c_n(\bar D_h^2)\right)), \qquad E\in \Ai\otimes M_4(\mathbb{C}),
\end{equation*} 
where $c_n(\bar D_h^2)\in \Ai\otimes \big(M_2(\mathbb{C})\oplus M_2(\mathbb{C})\big)\subset \Ai\otimes M_4(\mathbb{C})$ and $\tr$ denotes the matrix trace.
Inspired by \eqref{Riccizetagamma}, now  we  can define  the Ricci functional for the noncommutative two torus. 
\begin{definition}\label{riccifunctionalnc}
The { Ricci functional} of the modular de Rham spectral triple $(\A,\mathcal{H}, \bar D_h)$ is the functional on $\A\otimes M_2(\mathbb{C})$ defined as
\begin{equation*}
\ricfun(F)=a_2(\gamma \tilde{F},\bar D^2)=\zeta(0,\gamma\tilde{F},\bar {D}^2_h)+\Tr(\tr(F)\pr_0)-\Tr(F \pr_1),\qquad \tilde{F}=\tr(F)\oplus 0\oplus F.
\end{equation*}
Here $\pr_j$ denotes the orthogonal projection on the kernel of $\lap_{h,j}$, for $j=0,1$.
\end{definition}

\begin{lemma}\label{Riccidensitylemma} 
There exists an element $\ricden\in\Ai\otimes M_2(\mathbb{C})$ such that 
\begin{equation*}
\ricfun(F)=\frac{1}{\Im(\tau)}\varphi(\tr(F\ricden)e^{-h}), \quad F\in \Ai\otimes M_2(\mathbb{C}).
\end{equation*}
\end{lemma}
\begin{proof}
Because of the form of the Laplacian $\bar D^2_h$ \eqref{Lapisdirectsumtau}, for any $F\in \Ai\otimes M_2(\mathbb{C})$, we have
$$a_2(\gamma \tilde{F},\bar D^2 )=a_2(\tr(F),\lap_{h,0})-a_2(F,\lap_{h,1}).$$
On the other hand, $\tr(F)e^{-t\lap_{h,0}}=\tr(Fe^{-t\lap_{h,0}\otimes {\rm I}_2})$ for any $t>0$, and thus
$$a_2(\tr(F),\lap_{h,0})=a_2(F,\lap_{h,0}\otimes I_2).$$
As a result, we have
\begin{eqnarray*}
\ricden(F)&=&a_2(\tr(F),\lap_{h,0})-a_2(F,\lap_{h,1})\\
&=& \varphi\left(\tr \Big(F\big(c_2(\lap_{h,0})\otimes {\rm I}_2- c_2(\lap_{h,1})\big)\Big)\right)\\
&=& \frac{1}{\Im(\tau)} \varphi\left(\Im(\tau)\tr \Big(F\big(c_2(\lap_{h,0})\otimes {\rm I}_2- c_2(\lap_{h,1})\big)\Big)e^h e^{-h}\right).
\end{eqnarray*}
Hence, 
\begin{equation}\label{ricdenexists}
\ricden=\Im(\tau)\Big(c_2(\lap_{h,0})\otimes {\rm I}_2- c_2(\lap_{h,1})\Big)e^h.
\end{equation}
\end{proof}
\begin{definition}
The element $\ricden$ will be called the {Ricci density} 
of the modular spectral triple with dilaton $h$. 
\end{definition}

The terms $c_2(\lap_{h,j})$ can be computed as integrals of the terms of the symbol of parametrix of $\lap_{h,j}$ (see \S \ref{Heatkernelpseudo}). 
Since the operator $\lap_{h,1}$ is the first order perturbation of $\lap_\varphi^{(0,1)}$,  we will only need to compute the difference
$c_2(\lap_{h,1})-c_2(\lap_\varphi^{(0,1)})\otimes {\rm I}_2$. 
The terms $c_2(\lap_{h,0})=c_2(k\lap_0k)$ and $c_2(\lap_\varphi^{(0,1)})$ are computed in  both \cite[Theorem 3.2]{Connes-Moscovici2014} and \cite[Theorem 5.3]{Fathizadeh-Khalkhali2013}, and the difference is given by
\begin{eqnarray}\label{scalar}
   R^\gamma &=& \big(c_2(k\lap k)\otimes{\rm I}_2-c_2(\lap_\varphi^{(0,1)})\big) e^h\\\nonumber
   &=&-\frac{\pi}{\Im(\tau)}\Big(K_\gamma(\nabla)(\lap_0(\log k))+H_\gamma\big(\nabla_1,\nabla_2\big)\left(\square_\Re (\log k)\right)\\
   && \qquad\qquad\qquad\qquad\qquad\qquad+ S(\nabla_1,\nabla_2)
(\square_\Im (\log k))\Big)e^h.\notag
 \end{eqnarray}
Here,  
\begin{eqnarray*}
\square_\Re (\ell) &=&
(\delta_1(\ell))^2+ \Re(\tau)\left(\delta_1(\ell)\delta_2(\ell)+\delta_2(\ell)\delta_1(\ell)\right)+|\tau|^2 (\delta_2(\ell))^2,\\
\square_\Im(\ell)&=& i\Im(\tau)(\delta_1(\ell)\delta_2(\ell)-\delta_2(\ell)\delta_1(\ell))
\end{eqnarray*}
with $\ell = \log k.$
Moreover, 
$$
K_\gamma(u) = \, \frac{\frac 12+\frac{\sinh(u/2)}{u}}{\cosh^2(u/4)}, 
$$
\begin{align*}
\begin{split}
&H_\gamma(s,t)= \big(1-\cosh((s+t)/2)\big)\\
&*\frac{t (s+t) \cosh(s)-s (s+t) \cosh(t)+(s-t) (s+t+\sinh(s)+\sinh(t)-\sinh(s+t))}{s t (s+t) \sinh\left(\frac{s}{2}\right) \sinh\left(\frac{t}{2}\right) \sinh\left(\frac{s+t}{2}\right)^2},
\end{split}
\end{align*}
\begin{equation}\label{Sfunction}
S(s,t)=\frac{(s+t-t\, \cosh(s)-s\, \cosh(t)-\sinh(s)-\sinh(t)+\sinh(s+t))}{s\, t\left(\sinh\left(\frac{s}{2}\right) \sinh\left(\frac{t}{2}\right) \sinh\left(\frac{s+t}{2}\right)\right)}.
\end{equation}

Now we can state the main result of this paper, which gives the Ricci density of the modular de Rham spectral triple. 
The computations can be found in the next section.
\begin{theorem}\label{maintheorem}
Let $k=e^{h/2}$ where $h=h^*\in \Ai$ is the dilaton.
Then the Ricci density of the modular de Rham spectral triple with dilaton $h$ is given by 
\begin{equation*}
\ricden= \frac{\Im(\tau)}{4\pi^2}R^\gamma\otimes {\rm I}_2- \frac{1}{4\pi} S(\nabla_1,\nabla_2)\big([\delta_1(\log k),\delta_2(\log k)]\big)e^h\otimes \begin{pmatrix}
i\Im(\tau) & \Im(\tau)^2  \\
 -1& i\Im(\tau)
\end{pmatrix}\,.
\end{equation*}
\end{theorem}
\begin{remark}
In the commutative case, i.e. when $\theta=0$, the formula of the Ricci density $\ricden$ is given by 
$\lim_{(s,t)\to (0,0)} \ricden$.
As noted in \cite[Remark 5.4.]{Fathizadeh-Khalkhali2013},
$$\lim_{(s,t)\to (0,0)}R^\gamma=-\frac{\pi}{\Im(\tau)} \lap_0(\log k),$$
and because $[\delta_1(\log k),\delta_2(\log k)]=0$, we have
\begin{equation*}
\ricden|_{\theta=0}=\frac{-1}{4\pi} \lap_0(\log k)e^h\otimes {\rm I}_2.
\end{equation*}   
Considering the normalization of the classical case that comes from the heat kernel coefficients (see for example \eqref{differenceofterms}), this gives the formula for the Ricci operator in the classical case.
\end{remark}
\begin{remark}
For two dimensional manifolds, the Ricci tensor of the conformally perturbed metric $\tilde{g}=e^{-h}g$ is given by 
$$\tilde\ric=\ric+\frac{1}{2}\lap(h) g,$$
where $ \lap$  is the Laplacian on functions.
In particular, if $g$ is a flat metric then the Ricci tensor is equal to $\frac{1}{2}\lap(h) g.$
Hence the Ricci operator is a multiple of the identity map of which coefficient is the half of the scalar curvature $R=e^{h}\lap h$.

Unlike the commutative case, the Ricci density $\ricden$ of the noncommutative two torus  is not a multiple of the identity matrix any more. It is not even a symmetric matrix (with entries in $\Ai$). 
Indeed, it has nonzero off diagonal terms,  which are multiples of 
$S(\nabla_1,\nabla_2)([\delta_1(\log k),\delta_2(\log k)])$. 
This phenomenon, observed here for the first time, is obviously a consequence of the noncommutative nature of the space under investigation.   
\end{remark}

\section{Computation of the Ricci density for the noncommutative two torus}\label{computation}
Our main purpose  in this section  is to compute the Ricci density of the modular de Rham spectral triple. To perform this task, we employ the strategy developed in \cite{Connes-Cohen1992,Connes-Tretkoff2011}, which has also been used in \cite{Fathizadeh-Khalkhali2013,Connes-Moscovici2014}.

\subsection{Pseudodifferential calculus on $\A$}

In this section, we briefly review the notion of pseudodifferential calculus defined for a $C^\ast$-dynamical system $(A,\mathbb{R}^m,\alpha)$ by  Connes in \cite{Connes1980}. Recently, a twisted version of this theory was introduced in \cite{Lesch-Moscovici2016}.

Let  $(A,\mathbb{R}^2,\alpha)$ be a $C^\ast$-dynamical system. 
Let  $A^\infty$ denote the algebra of smooth elements of $A$, i.e. these elements $a\in A$ for which the $A$-values function $\mathbb{R}^2 \ni s\to \alpha_s(a)$ is smooth.
We denote by $V:\mathbb{R}^2\to M(A\rtimes \mathbb{R}^2)$ the canonical unitary representation of $\mathbb{R}^2$ on the multiplier algebra of the crossed product algebra.
For a given non-negative integer multi index $\alpha= (\alpha_1, \alpha_2)$, we use the following notation:
\begin{equation*}
\partial^{\alpha}=\frac{\partial^{\alpha_1}}{\partial \xi_1^{\alpha_1}}  \frac{\partial^{\alpha_2}}{\partial \xi_2^{\alpha_2}},\qquad \delta^{\alpha}= \delta_1^{\alpha_1}\delta_2^{\alpha_2},
\end{equation*}
where $\delta_j$ is the infinitesimal generator of the $\alpha$ in the direction of $e_j$.

A smooth map $\rho: \mathbb{R}^2 \to A^{\infty}$ is called
 a { symbol of order $d$}, if 
 \begin{enumerate}
 \item[(i)] For every non-negative integers $i_1, i_2, j_1,
j_2,$ there exists a constant $C$ such that
\begin{equation*}
 \|\delta^{(i_1, i_2)}  \partial^{(j_1, j_2)}  \rho(\xi) \|
\leq C (1+|\xi|)^{d-j_1-j_2}.
\end{equation*}
\item [(ii)]
There exists a smooth map $k: \mathbb{R}^2 \to
A^{\infty}$ such that
$$\lim_{\lambda \to \infty} \lambda^{-d} \rho(\lambda\xi_1, \lambda\xi_2) = k (\xi_1, \xi_2).$$
\end{enumerate}
The space of symbols of order $d$ is denoted by $S^d(A^\infty)$.

To every symbol $\rho\in S^d(A^\infty)$, one can assign the pseudodifferential multiplier operator $P_\rho:\mathcal{S}(\mathbb{R}^2,A^\infty)\to \mathcal{S}(\mathbb{R}^2,A^\infty)$ acting on the space $\mathcal{S}(\mathbb{R}^2,A^\infty)$ of all $A^\infty$-valued Schwartz functions, defined by
\begin{equation}\label{pseudooperator}
P_\rho=\int_{\mathbb{R}^n}\check{\rho}(\xi)V_\xi d\xi.
\end{equation}
Here $d\xi=(2\pi)^{-2}d_L\xi$, where $d_L\xi$ is the Lebesgue measure on $\mathbb{R}^2$, and $\check{\rho}$ denotes the inverse Fourier transform of $\rho$. 
In \cite[Proposition 8]{Connes1980}, it was shown that the product of two pseudodifferential multipliers $P_{\rho_1}$ and $P_{\rho_2}$, with $\rho_j\in S^{d_j}(A^\infty)$, $j=1,2$, is also a pseudodifferential multiplier with symbol $\rho\in S^{d_1+d_2}(A^\infty)$. 
The product symbol $\rho$ is given asymptotically as follows:
\begin{equation}\label{productsymbol}
\rho\sim \sum_{\alpha=(\alpha_1,\alpha_2) \geq 0} 
\frac{1}{\alpha_1!\alpha_2!}\partial^{\alpha} (\rho_1)
\delta^{\alpha}(\rho_2).
\end{equation}

For a covariant representation $(U,\pi)$ of the $C^*$-dynamical system $(A,\mathbb{R}^2,\alpha)$ on a Hilbert space $\mathcal{H}$, the pseudodifferential multiplier $P_\rho$ induces a densely-defined differential operator on $\mathcal{H}$, denoted by the same letter $P_\rho$, by simply replacing $V_\xi$ with $U_\xi$ in \eqref{pseudooperator}. In particular, if $\omega$ is an $\alpha$-invariant state of $A$, then the left regular  representation of $A$ on the GNS Hilbert space  $\mathcal{H}_\omega$ gives rise to a covariant representation of $(A,\mathbb{R}^2,\alpha)$.
Then the pseudodifferential operator $P_\rho$ on $\mathcal{H}_\omega$ is defined  on the elements of $A^\infty$ as     
\begin{equation*}
P_\rho(a)=\int_{\mathbb{R}^2}e^{is\xi}\rho (\xi)\alpha_{\xi}(a)d\xi.
\end{equation*}
The product of symbols of pseudodifferential operators with $A^\infty$-valued symbols on $\mathcal{H}$ is also given by \eqref{productsymbol}.  

In this work, we shall use the symbol calculus of the pseudodifferential operators of the $C^\ast$-dynamical system $(A_\theta \otimes M_2(\mathbb{C}), \mathbb{R}^2,\alpha\otimes {\rm Id})$ on $\mathcal{H}_0\otimes \mathbb{C}^2$, where $\alpha$ is the action \eqref{actionofR2onA}.

\subsection{Heat kernel coefficients via pseudodifferential calculus}\label{Heatkernelpseudo}
The asymptotic expansion of the  trace  of the heat kernel can be computed using the symbol calculus. 
In the classical case, this technique was effectively used in the heat kernel proof of the index theorem (cf. 
\cite[Section 1.8]{Gilkey1995}). In the noncommutative case, it was used for local computations of spectral invariants for the noncommutative two torus in \cite{Connes-Cohen1992,Connes-Tretkoff2011,Connes-Moscovici2014,Fathizadeh-Khalkhali2012}. 
For the sake of completeness, we outline the main steps of this technique, and then we use it to compute the Ricci density for the noncommutative two torus.

Let $P$ be a positive differential operator of order two with the symbol $\sigma(P)$, written as 
$$\sigma(P)=a_2+a_1+a_0\in S^2(\Ai\otimes M_2(\mathbb{C})),$$ 
where each term $a_j$ is homogeneous of order $j$. 
Moreover we assume that $P$ is elliptic, that is $a_2(\xi)$ is invertible for all non-zero $\xi\in\mathbb{R}^2$.  
For every $\lambda\in\mathbb{C}\setminus\mathbb{R}^{\geq 0}$, the parametrix $(P-\lambda)^{-1}$ of $P$ is a pseudodifferential operator of order $-2$ and its symbol  can be expressed as
\begin{equation}\label{parametrixsymbol}
\sigma\left((P-\lambda)^{-1}\right)=b_0(\xi,\lambda)+b_1(\xi,\lambda)+b_2(\xi,\lambda)+\cdots,
\end{equation}
where $b_n(\xi,\lambda)$ is  $(-n-2,1)$-homogeneous in $(\xi,\lambda)$ (cf. \cite{Gilkey1995}).
The symbol product \eqref{productsymbol} resulted from the equation $(P-\lambda)(P-\lambda)^{-1}= I$, up to a pseudodifferential operator of order $-\infty$, provides a recursive formula that allows us to find $b_n(\xi,\lambda)$ \cite{Connes-Cohen1992,Connes-Tretkoff2011}.
For instance,
\begin{eqnarray}\label{parametrixsymbols}
\nonumber b_0(\xi,\lambda)&=&(a_2(\xi)-\lambda)^{-1},\\ \nonumber
b_1(\xi,\lambda)&=&-\Big(b_0a_1b_0+\partial_1(b_0)\delta_1(a_2)b_0+\partial_2(b_0)\delta_2(a_2)b_0\Big),\\
b_2(\xi,\lambda)&=&-\Big(b_0a_0b_0+b_1a_1b_0+\partial_1(b_0)\delta_1(a_1)b_0+\partial_2(b_0)\delta_2(a_1)b_0 \\ \nonumber 
&& +\partial_1(b_1)\delta_1(a_2)b_0 +\partial_2(b_1)\delta_2(a_2)b_0+\partial_1\partial_2(b_0)\delta_1\delta_2(a_2)b_0\\ \nonumber 
&& +\frac{1}{2}\partial_1^2(b_0)\delta^2_1(a_2)b_0+\frac{1}{2}\partial_2^2(b_0)\delta^2_2(a_2)b_0\Big).\nonumber 
\end{eqnarray}

The operator $e^{-tP}$ is a pseudodifferential operator of order $-\infty$. 
It can also be defined using the Cauchy formula in terms of the parametrix,
\begin{equation*}
e^{-tP}=\frac{1}{2\pi i}\int_\gamma e^{-t\lambda}(P-\lambda)^{-1}d\lambda,
\end{equation*} 
where $\gamma$ is a clockwise-oriented contour around the non-negative part of real axis.
The above formula provides us with the asymptotic expansion formula of the symbol as $t\to 0^+$:
\begin{equation*}
\sigma(e^{-tP})\sim \sum_{n=0}^\infty t^{\frac{n-2}{2}}b_n(\xi),
\end{equation*}
where
$$b_n(\xi)=\frac{1}{2\pi i}\int_\gamma e^{-t \lambda} b_n(\xi,\lambda) d\lambda.$$
Note that every $F\in \Ai\otimes M_2(\mathbb{C})$ is a differential operator of order zero. By fixing such a $F$, one can readily see that 
$$\sigma(Fe^{-tP})\sim \sum_{n=0}^\infty t^{\frac{n-2}{2}} F b_n(\xi).$$
On the other hand, $Fe^{-tP}$ is a trace class operator and its trace is given by $\Tr\big(Fe^{-tP}\big)=\int_{\mathbb{R}^2}\varphi\big(F\sigma(e^{-tP})\big)d\xi$. 
Therefore, the asymptotic expansion of the localized heat trace $\Tr(Fe^{-tP})$ is given by 
\begin{equation*}
\Tr(Fe^{-tP})\sim \sum_{n=0}^\infty t^{\frac{n-2}{2}} \int_{\mathbb{R}^2}\varphi\Big(\tr\big(Fb_n(\xi)\big)\Big) d\xi .
\end{equation*}
The coefficient of the $t^{-1+n/2}$, denoted by $a_n(F,P)$, is given by  
\begin{equation}\label{localformofterms}
a_n(F,P)=\varphi\big(\tr(Fc_n(P))\big),\;\mbox{where}\;\, c_n(P)= \frac{1}{2\pi i}\int_{\mathbb{R}^2}\int_\gamma e^{-\lambda}  b_n(\xi,\lambda) d\lambda d\xi.
\end{equation}
For $n=2$, it can be shown by using a homogeneity argument (cf. \cite[Section 6]{Connes-Moscovici2014}) that 
\begin{equation}\label{locala2}
c_2(P)=\int_{\mathbb{R}^2} b_2(\xi,-1)d\xi.
\end{equation} 
In the next section, we will use these computations when the operator $P$ is one of the Laplacians $\lap_{h,j}$ or $\lap_\varphi^{(0,1)}$.

\subsection{Computation of the Ricci density}
The operators $\lap_{h,0}$ and $\lap_{h,1}$ are pseudodifferential operators associated with the dynamical systems $(\A,\mathbb{R}^2,\alpha)$ and  
$(\A\otimes M_2(\mathbb{C}),\mathbb{R}^2,\alpha\otimes {\rm Id})$ respectively, represented on the covariant representation spaces $\Hcal_0$ and $\Hcal_0\oplus \Hcal_0$.
In this section, we will perform the computation \eqref{locala2} for the Laplacians $\lap_{h,1}$, and then borrowing a computation done in \cite{Fathizadeh-Khalkhali2013,Connes-Moscovici2014}, we will find the Ricci density as it is given in Theorem \ref{maintheorem}. 
First of all, we describe the relation between the term $b_2(\xi,\lambda)$ for $\lap_{h,1}$ and the corresponding term for $\lap_\varphi^{(0,1)}$.
Let us start with the symbol of $\lap_{h,1}$, which can easily be computed using \eqref{lap1}, as follows.

\begin{lemma} 
The symbol $\sigma(\lap_{h,1})$ of the operator $\lap_{h,1}$ is the sum of homogeneous terms $a_j(\xi)$, $j=0,1,2$, given by
$$
a_2(\xi)=k^2\big(\xi_1^2+2\Re(\tau)\xi_1\xi_2+|\tau|^2\xi_2^2\big)\otimes{\rm I}_2,\quad
a_0(\xi)= 0,$$
\begin{equation*}
{\footnotesize a_1(\xi)
=\begin{pmatrix}
\delta_1(k^2)\xi_1+2\Re(\tau)\delta_1(k^2)\xi_2+|\tau|^2 \delta_2(k^2)\xi_2&
|\tau|^2(\delta_1(k^2)\xi_2-\delta_1(k^2)\xi_1)\\ &\\
\delta_1(k^2)\xi_1-\delta_1(k^2)\xi_2 &\hspace{-0.5cm}
\delta_1(k^2)\xi_1+2\Re(\tau)\delta_2(k^2)\xi_1+ |\tau|^2 \delta_2(k^2)\xi_2
\end{pmatrix}}.
\end{equation*}
\end{lemma}
On the other hand, if $\sigma(\lap_\varphi^{(0,1)})=a_2'(\xi)+a_1'(\xi)+a_0'(\xi)$, where $a_j'$ is the homogeneous term of order $j$, then we infer from \eqref{dallap-deRhamlap} that
\begin{eqnarray*}
a_2&=&a'_2 \otimes {\rm I}_2\\
a_1 &=& a'_1 \otimes  {\rm I}_2 +a''_1 \otimes \sigma\\
a_0&=& a'_0\otimes {\rm I}_2=0.
\end{eqnarray*}
Here
$$a_1''(\xi)=\delta_1(k^2) \xi_2-\delta_2(k^2)\xi_1,$$
and using \cite[Lemma 6.1]{Connes-Moscovici2014}, $a_1'(\xi)$ is given by 
$$a_1'(\xi)=\big((\delta_1(k^2)+\tau\delta_2(k^2)\big)(\xi_1+\bar\tau\xi_2).$$

Let us denote the terms of the symbol of the parametrix $(\lap_{h,1}-\lambda)^{-1}$ by $b_n(\xi,\lambda)$, as in \eqref{parametrixsymbol}, i.e.  
$$\sigma\big((\lap_{h,1}-\lambda)^{-1}\big)=b_0(\xi,\lambda)+b_1(\xi,\lambda)+b_2(\xi,\lambda)+\cdots.$$ 
Similarly, we denote the terms of the symbol of the parametrix of $\lap_\varphi^{(0,1)}$ by $b_n'(\xi,\lambda)$, i.e. 
$$\sigma\big((\lap_\varphi^{(0,1)}-\lambda)^{-1}\big)=b_0'(\xi,\lambda)+b_1'(\xi,\lambda)+b_2'(\xi,\lambda)+\cdots\, . $$
Note that both $b_n$ and $b_n'$ are given by \eqref{parametrixsymbols}, using the symbols of $\lap_{h,1}$ and $\lap_\varphi^{(0,1)}$ respectively.

In the following lines, we will show that each term $b_n(\xi,\lambda)$, $n=0,1,2$, is the sum of $b_n'(\xi,\lambda)\otimes {\rm I}_2$ and another term, which will be computed explicitly.
For the first term, we have
$$b_0=b_0'\otimes {\rm I}_2=(k^2|\xi|^2-\lambda)^{-1}\otimes {\rm I}_2.$$
Note that we have used the identity $a'_2\otimes {\rm I}_2=a_2$. Furthermore,
\begin{eqnarray*}
b_1&=&-\left(b_0a_1b_0+\partial_1(b_0)\delta_1(a_2)b_0+\partial_2(b_0)\delta_2(a_2)b_0\right)\\
&=& b_1' \otimes {\rm I}_2-b_0a_1'' b_0\sigma.
\end{eqnarray*}
Finally, the second term $b_2$ is given by 
\begin{eqnarray*}
b_2&=&-\big(b_0a_0b_0+b_1a_1b_0+\partial_1(b_0)\delta_1(a_1)b_0+\partial_2(b_0)\delta_2(a_1)b_0 \\
&& +\partial_1(b_1)\delta_1(a_2)b_0 +\partial_2(b_1)\delta_2(a_2)b_0+\partial_1\partial_2(b_0)\delta_1\delta_2(a_2)b_0\\
&& +\frac{1}{2}\partial_1\partial_1(b_0)\delta^2_1(a_2)b_0+\frac{1}{2}\partial_2\partial_2(b_0)\delta^2_2(a_2)b_0\big)\\
&&\\
&=& b_2'\otimes  {\rm I}_2
-\Big(    
b_1'a_1''b_0
-2i\Im(\tau) b_0a_1'' b_0a_1''b_0
-b_0a_1''b_0a_1'b_0+\partial_1(b_0)\delta_1(a_1'')b_0 \\ && +\partial_2(b_0)\delta_2(a_1'')b_0 -\partial_1(b_0a_1''b_0)\delta_1(a_2)b_0 -\partial_2(b_0a_1''b_0)\delta_2(a_2)b_0\Big)\sigma.
\end{eqnarray*} 
By taking
\begin{eqnarray}\label{b2doubleprime}
\begin{aligned}
b_2''&= b_0a'_1b_0a_1''b_0+
\partial_1(b_0)\delta_1(a_2')b_0a_1''b_0
+\partial_2(b_0)\delta_2(a_2')b_0a_1''b_0
\\
&+2i\Im(\tau) b_0a_1'' b_0a_1''b_0
+b_0a_1''b_0a_1'b_0  
-\partial_1(b_0)\delta_1(a_1'')b_0
-\partial_2(b_0)\delta_2(a_1'')b_0
\\
&+\partial_1(b_0a_1''b_0)\delta_1(a_2)b_0 +\partial_2(b_0a_1''b_0)\delta_2(a_2)b_0,
\end{aligned}
\end{eqnarray}
and using the identity $\sigma^2=2i\Im(\tau)\sigma$, we then have
\begin{eqnarray*}
b_2&=&b_2'\otimes  {\rm I}_2
+ b_2'' \sigma.
\end{eqnarray*}
\begin{lemma} 
The Ricci density for the modular de Rham spectral triple with dilaton $h$ is given by
\begin{eqnarray*}
\ricden= \Im(\tau) \big(\frac{1}{4\pi^2}R^\gamma\otimes {\rm I}_2 - \int_{\mathbb{R}^2}b_2''(\xi,-1) d\xi \otimes \sigma\big) e^h.
\end{eqnarray*}
\end{lemma}
\begin{proof}
We have already seen in \eqref{ricdenexists} that the Ricci density can be described as
\begin{equation*}
\ricden =\Im(\tau) \big(c_2(\lap_{h,0})\otimes {\rm I}_2-c_2(\lap_{h,1})\big) e^h.
\end{equation*}
By \eqref{localformofterms}, we then have
\begin{eqnarray*}
c_2(\lap_{h,0})\otimes {\rm I}_2-c_2(\lap_{h,1}) &=& \frac{1}{2\pi i}\int_{\mathbb{R}^2}\int_\gamma e^{-\lambda}  b_2^{\lap_{h,0}}(\xi,\lambda) \otimes {\rm I}_2  - b_2(\xi,\lambda) \, d\lambda d\xi \\
&=& \frac{1}{2\pi i}\int_{\mathbb{R}^2}\int_\gamma e^{-\lambda}  \Big(b_2^{k\lap_0 k}(\xi,\lambda)   - b'_2(\xi,\lambda)\Big)\otimes {\rm I}_2- b_2''\otimes \sigma\, d\lambda d\xi\\
&=& \int_{\mathbb{R}^2}  \big(b_2^{k\lap_0 k}(\xi,-1)   - b'_2(\xi,-1)\big)\otimes {\rm I}_2- b_2''(\xi,-1)\otimes \sigma  d\xi.
\end{eqnarray*} 
The last identity was deduced from the same homogeneity argument used to prove \eqref{locala2}. On the other hand,
the term $\int_{\mathbb{R}^2}  b_2^{k\lap_0 k}(\xi,-1)   - b'_2(\xi,-1)d\xi$ was computed in \cite{Connes-Moscovici2014,Fathizadeh-Khalkhali2013}, where it was shown to be equal to $\frac{1}{4\pi^2}R^\gamma$. 
Therefore we have   
\begin{equation}
c_2(\lap_{h,0})\otimes {\rm I}_2-c_2(\lap_{h,1}) = 
\frac{1}{4\pi^2}R^\gamma\otimes {\rm I}_2 -\int_{\mathbb{R}^2} b_2''(\xi,-1)\otimes \sigma \, d\xi.
\end{equation} 
This completes the proof of the lemma.
\end{proof}
In the rest of this section, we will carry out the computation needed to obtain a formula for $\int_{\mathbb{R}^2}b_2''(\xi,-1)d\xi$. For this purpose, by replacing all the terms in \eqref{b2doubleprime}, one can obtain a formula for $b_2''(\xi,\lambda)$ in terms of $b_0$ as follows:
\allowdisplaybreaks
\begin{eqnarray*}
b_2''(\xi,\lambda)&=& 2 \xi_1 \xi_2 k^2  b_0^2  \delta_1\left( \delta_1(k^2)\right)  b_0+2 \xi_2^2 \tau_1 k^2  b_0^2  \delta_1\left( \delta_1(k^2)\right)  b_0\\
&-&2 \xi_1^2 k^2  b_0^2  \delta_1\left( \delta_2(k^2)\right)  b_0-2 \xi_1 \xi_2 \tau_1 k^2  b_0^2  \delta_1\left( \delta_2(k^2)\right)  b_0\\
&+&2 \xi_1 \xi_2 \tau_1 k^2  b_0^2  \delta_2\left( \delta_1(k^2)\right)  b_0+2 \xi_2^2 \tau_1^2 k^2  b_0^2  \delta_2\left( \delta_1(k^2)\right)  b_0\\
&+&2 \xi_2^2 \tau_2^2 k^2  b_0^2  \delta_2\left( \delta_1(k^2)\right)  b_0-2 \xi_1^2 \tau_1 k^2  b_0^2  \delta_2\left( \delta_2(k^2)\right)  b_0\\
&-&2 \xi_1 \xi_2 \tau_1^2 k^2  b_0^2  \delta_2\left( \delta_2(k^2)\right)  b_0-2 \xi_1 \xi_2 \tau_2^2 k^2  b_0^2  \delta_2\left( \delta_2(k^2)\right)  b_0\\
&+&2 \xi_1 \xi_2  b_0 \delta_1(k^2)  b_0 \delta_1(k^2)  b_0+2 \xi_2^2 \tau_1  b_0 \delta_1(k^2)  b_0 \delta_1(k^2)  b_0\\
&+&2 \xi_1 \xi_2 \tau_1  b_0 \delta_1(k^2)  b_0 \delta_2(k^2)  b_0+2 \xi_2^2 \tau_1^2  b_0 \delta_1(k^2)  b_0 \delta_2(k^2)  b_0\\
&+&2 \xi_2^2 \tau_2^2  b_0 \delta_1(k^2)  b_0 \delta_2(k^2)  b_0-2 \xi_1^2  b_0 \delta_2(k^2)  b_0 \delta_1(k^2)  b_0\\
&-&2 \xi_1 \xi_2 \tau_1  b_0 \delta_2(k^2)  b_0 \delta_1(k^2)  b_0-2 \xi_1^2 \tau_1  b_0 \delta_2(k^2)  b_0 \delta_2(k^2)  b_0\\
&-&2 \xi_1 \xi_2\tau_1^2  b_0 \delta_2(k^2)  b_0 \delta_2(k^2)  b_0-2 \xi_1 \xi_2\tau_2^2  b_0 \delta_2(k^2)  b_0 \delta_2(k^2)  b_0\\
&-&2 \xi_1^3 \xi_2  b_0 \delta_1(k^2) k^2  b_0^2 \delta_1(k^2)  b_0-6 \xi_1^2 \xi_2^2 \tau_1  b_0 \delta_1(k^2) k^2  b_0^2 \delta_1(k^2)  b_0\\
&-&6 \xi_1 \xi_2^3 \tau_1^2  b_0 \delta_1(k^2) k^2  b_0^2 \delta_1(k^2)  b_0-2 \xi_2^4 \tau_1^3  b_0 \delta_1(k^2) k^2  b_0^2 \delta_1(k^2)  b_0\\
&-&2 \xi_1 \xi_2^3 \tau_2^2  b_0 \delta_1(k^2) k^2  b_0^2 \delta_1(k^2)  b_0-2 \xi_2^4 \tau_1 \tau_2^2  b_0 \delta_1(k^2) k^2  b_0^2 \delta_1(k^2)  b_0\\
&-&2 \xi_1^3 \xi_2 \tau_1  b_0 \delta_1(k^2) k^2  b_0^2 \delta_2(k^2)  b_0-6 \xi_1^2 \xi_2^2 \tau_1^2  b_0 \delta_1(k^2) k^2  b_0^2 \delta_2(k^2)  b_0\\
&-&6 \xi_1 \xi_2^3 \tau_1^3  b_0 \delta_1(k^2) k^2  b_0^2 \delta_2(k^2)  b_0-2 \xi_2^4 \tau_1^4  b_0 \delta_1(k^2) k^2  b_0^2 \delta_2(k^2)  b_0\\
&-&2 \xi_1^2 \xi_2^2 \tau_2^2  b_0 \delta_1(k^2) k^2  b_0^2 \delta_2(k^2)  b_0-6 \xi_1 \xi_2^3 \tau_1 \tau_2^2  b_0 \delta_1(k^2) k^2  b_0^2 \delta_2(k^2)  b_0\\
&-&4 \xi_2^4 \tau_1^2 \tau_2^2  b_0 \delta_1(k^2) k^2  b_0^2 \delta_2(k^2)  b_0-2 \xi_2^4 \tau_2^4  b_0 \delta_1(k^2) k^2  b_0^2 \delta_2(k^2)  b_0\\
&+&2 \xi_1^4  b_0 \delta_2(k^2) k^2  b_0^2 \delta_1(k^2)  b_0+6 \xi_1^3 \xi_2 \tau_1  b_0 \delta_2(k^2) k^2  b_0^2 \delta_1(k^2)  b_0\\
&+& 6 \xi_1^2 \xi_2^2 \tau_1^2  b_0 \delta_2(k^2) k^2  b_0^2 \delta_1(k^2)  b_0+2 \xi_1 \xi_2^3 \tau_1^3  b_0 \delta_2(k^2) k^2  b_0^2 \delta_1(k^2)  b_0\\
&+& 2 \xi_1^2 \xi_2^2 \tau_2^2  b_0 \delta_2(k^2) k^2  b_0^2 \delta_1(k^2)  b_0+2 \xi_1 \xi_2^3 \tau_1 \tau_2^2  b_0 \delta_2(k^2) k^2  b_0^2 \delta_1(k^2)  b_0\\
&+&2 \xi_1^4\tau_1  b_0 \delta_2(k^2) k^2  b_0^2 \delta_2(k^2)  b_0+6 \xi_1^3 \xi_2\tau_1^2  b_0 \delta_2(k^2) k^2  b_0^2 \delta_2(k^2)  b_0\\
&+& 6 \xi_1^2 \xi_2^2\tau_1^3  b_0 \delta_2(k^2) k^2  b_0^2 \delta_2(k^2)  b_0+2 \xi_1 \xi_2^3\tau_1^4  b_0 \delta_2(k^2) k^2  b_0^2 \delta_2(k^2)  b_0\\
&+&2 \xi_1^3 \xi_2\tau_2^2  b_0 \delta_2(k^2) k^2  b_0^2 \delta_2(k^2)  b_0+6 \xi_1^2 \xi_2^2\tau_1\tau_2^2  b_0 \delta_2(k^2) k^2  b_0^2 \delta_2(k^2)  b_0\\
&+& 4 \xi_1 \xi_2^3\tau_1^2\tau_2^2  b_0 \delta_2(k^2) k^2  b_0^2 \delta_2(k^2)  b_0+2 \xi_1 \xi_2^3\tau_2^4  b_0 \delta_2(k^2) k^2  b_0^2 \delta_2(k^2)  b_0\\
&-&4 \xi_1^3 \xi_2 k^2  b_0^2 \delta_1(k^2)  b_0 \delta_1(k^2)  b_0-12 \xi_1^2 \xi_2^2\tau_1 k^2  b_0^2 \delta_1(k^2)  b_0 \delta_1(k^2)  b_0\\
&-&12 \xi_1 \xi_2^3 \tau_1^2 k^2  b_0^2 \delta_1(k^2)  b_0 \delta_1(k^2)  b_0-4 \xi_2^4 \tau_1^3 k^2  b_0^2 \delta_1(k^2)  b_0 \delta_1(k^2)  b_0\\
&-&4 \xi_1 \xi_2^3 \tau_2^2 k^2  b_0^2 \delta_1(k^2)  b_0 \delta_1(k^2)  b_0-4 \xi_2^4 \tau_1 \tau_2^2 k^2  b_0^2 \delta_1(k^2)  b_0 \delta_1(k^2)  b_0\\
&+& 2 \xi_1^4 k^2  b_0^2 \delta_1(k^2)  b_0 \delta_2(k^2)  b_0+4 \xi_1^3 \xi_2 \tau_1 k^2  b_0^2 \delta_1(k^2)  b_0 \delta_2(k^2)  b_0\\
&-&4 \xi_1 \xi_2^3 \tau_1^3 k^2  b_0^2 \delta_1(k^2)  b_0 \delta_2(k^2)  b_0-2 \xi_2^4 \tau_1^4 k^2  b_0^2 \delta_1(k^2)  b_0 \delta_2(k^2)  b_0\\
&-&4 \xi_1 \xi_2^3 \tau_1 \tau_2^2 k^2  b_0^2 \delta_1(k^2)  b_0 \delta_2(k^2)  b_0-4 \xi_2^4 \tau_1^2 \tau_2^2 k^2  b_0^2 \delta_1(k^2)  b_0 \delta_2(k^2)  b_0\\
&+& 2 \xi_1^4 k^2  b_0^2 \delta_2(k^2)  b_0 \delta_1(k^2)  b_0-2 \xi_2^4 \tau_2^4 k^2  b_0^2 \delta_1(k^2)  b_0 \delta_2(k^2)  b_0\\
&+& 4 \xi_1^3 \xi_2 \tau_1 k^2  b_0^2 \delta_2(k^2)  b_0 \delta_1(k^2)  b_0-4 \xi_1 \xi_2^3 \tau_1^3 k^2  b_0^2 \delta_2(k^2)  b_0 \delta_1(k^2)  b_0\\
&-& 4 \xi_1 \xi_2^3 \tau_1 \tau_2^2 k^2  b_0^2 \delta_2(k^2)  b_0 \delta_1(k^2)  b_0-2 \xi_2^4 \tau_1^4 k^2  b_0^2 \delta_2(k^2)  b_0 \delta_1(k^2)  b_0\\
&-&4 \xi_2^4 \tau_1^2 \tau_2^2 k^2  b_0^2 \delta_2(k^2)  b_0 \delta_1(k^2)  b_0-2 \xi_2^4 \tau_2^4 k^2  b_0^2 \delta_2(k^2)  b_0 \delta_1(k^2)  b_0\\
&+& 4 \xi_1^4 \tau_1 k^2  b_0^2 \delta_2(k^2)  b_0 \delta_2(k^2)  b_0+12 \xi_1^3 \xi_2 \tau_1^2 k^2  b_0^2 \delta_2(k^2)  b_0 \delta_2(k^2)  b_0\\
&+& 12 \xi_1^2 \xi_2^2 \tau_1^3 k^2  b_0^2 \delta_2(k^2)  b_0 \delta_2(k^2)  b_0+4 \xi_1 \xi_2^3 \tau_1^4 k^2  b_0^2 \delta_2(k^2)  b_0 \delta_2(k^2)  b_0\\
&+& 4 \xi_1^3 \xi_2 \tau_2^2 k^2  b_0^2 \delta_2(k^2)  b_0 \delta_2(k^2)  b_0+12 \xi_1^2 \xi_2^2 \tau_1 \tau_2^2 k^2  b_0^2 \delta_2(k^2)  b_0 \delta_2(k^2)  b_0\\
&+& 8 \xi_1 \xi_2^3 \tau_1^2 \tau_2^2 k^2  b_0^2 \delta_2(k^2)  b_0 \delta_2(k^2)  b_0+4 \xi_1 \xi_2^3 \tau_2^4 k^2  b_0^2 \delta_2(k^2)  b_0 \delta_2(k^2)  b_0\,.
\end{eqnarray*}
\allowdisplaybreaks[0]
In the above formula $\tau_1=\Re(\tau)$ and $\tau_2=\Im(\tau)$. 
We  then change the variables $(\xi_1,\xi_2)$ to the new variables $(r,\theta)$  via the transformation 
$$\xi_1=r \cos(\theta )-\frac{r \Re(\tau)}{\Im(\tau)} \sin (\theta),\qquad \xi_2=\frac{r}{\Im(\tau)} \sin (\theta ).$$
This change of variable will change the volume form as follows:
$$d\xi=(2\pi)^{-2}d\xi_1d\xi_2=\frac{1}{(2\pi)^2\Im(\tau)} rdrd\theta.$$
 The integral $b_2''(\xi,-1)$ with respect to $\theta\in [0,2\pi]$ is equal to  
 \begin{eqnarray*}
 \frac{r}{(2\pi)^2 \Im(\tau)}\int_0^{2\pi}b_2''(\xi,-1)d\theta &=&
 \frac{ r}{2\pi \tau_2}\Big(  
 r^2b_0\delta_1(k^2)b_0\delta_2(k^2)b_0
 - r^2b_0\delta_2(k^2)b_0\delta_1(k^2)\\
&&\quad -r^4b_0\delta_1(k^2)k^2b_0^2\delta_2(k^2)b_0
 +r^4b_0\delta_2(k^2)k^2b_0^2\delta_1(k^2)b_0  \Big).
 \end{eqnarray*}
  
  We first write these expressions in terms of $h/2=\log k$.
To do so, we will use the following identities from \cite{Connes-Moscovici2014}:
\begin{equation}
\delta_j(k)=kf(\Delta)(\delta_j(\log(k))),\quad f(u)=\frac{2(-1+\sqrt{u})}{\log(u)},
\end{equation}
and also  $\Delta(ka)=k\Delta(a)$ and $ak^n=k^n\Delta^{\frac{n}{2}}(a)$. 
Then we have,
\begin{eqnarray*}
\delta_j(k^2)&=& k\delta_j(k)+\delta_j(k)k\\
&=& k\delta_j(k)+k\Delta^{\frac12}\big(\delta_j(k)\big)\\
&=& k(1+\Delta^{\frac12})\big(\delta_j(k)\big)\\
&=& k(1+\Delta^{\frac12})\big(kf(\Delta)(\delta_j(\log(k))\big)\\
&=& k^2(1+\Delta^{\frac12})\big(f(\Delta)(\delta_j(\log(k))\big)\\
&=& k^2g(\Delta)\big(\delta_j(\log(k)\big).
\end{eqnarray*}
where $g(u)=\frac{2(u-1)}{\log(u)}.$
We will also use the notation $g_j(u)=u^jg(u)$. \\
There are only two different types of terms to be integrated with respect to the radial variable $r$. To perform the integration, we will use the rearrangement lemma \cite[Lemma 6.2]{Connes-Moscovici2014}, which gives us
\begin{eqnarray*}
&&\int_0^\infty r^3 b_0\delta_j(k^2)b_0\delta_{j'}(k^2)b_0dr \\
&=& \int_0^\infty r^3 b_0\delta_j(k^2)b_0k^2g(\Delta)\big(\delta_{j'}(\log(k))\big)b_0dr\\
&=& \int_0^\infty r^3 b_0k^2 
g(\Delta)\big(\delta_j(\log(k))\big) k^2
b_0
g(\Delta)\big(\delta_{j'}(\log(k))\big) b_0dr\\
&=& \int_0^\infty r^3 b_0k^4
g_1(\Delta)\big(\delta_j(\log(k))\big) b_0 
g(\Delta)\big(\delta_{j'}(\log(k))\big) b_0dr\\
&=& \frac{1}{2}F_{1,1,1}(\Delta_1,\Delta_2)\Big(g_1(\Delta)\big(\delta_{j}(\log(k))\big)\, g(\Delta)\big(\delta_{j'}(\log(k))\big)\Big).
\end{eqnarray*}
Similarly,
\begin{eqnarray*}
&&\int_0^\infty r^5b_0\delta_j(k^2)k^2b_0^2\delta_{j'}(k^2)b_0dr\\
&=& \int_0^\infty r^5b_0 k^2g(\Delta)\big(\delta_{j}(\log(k))\big) k^2 b_0^2k^2 g(\Delta)\big(\delta_{j'}(\log(k))\big)b_0dr\\
&=& \int_0^\infty r^5k^6 b_0 \Delta^2 g(\Delta)\big(\delta_{j}(\log(k))\big) b_0^2g(\Delta)(\delta_{j'}(\log(k))\big)b_0dr\\
 &=& \frac12 F_{1,2,1}(\Delta_1,\Delta_2)\Big(
 g_2(\Delta)\big(\delta_j(\log(k))\big)
 g(\Delta)\big( \delta_{j'}(\log(k))\big)\Big).
\end{eqnarray*}
  
Note that 
  \begin{eqnarray*}
&&F_{111}(e^s,e^t)g_1(e^s)g(e^t)-F_{121}(e^s,e^t)g_2(e^s)g(e^t)\\
&=& \frac{(s+t-t\, \cosh(s)-s\, \cosh(t)-\sinh(s)-\sinh(t)+\sinh(s+t))}{s\, t\left(\sinh\left(\frac{s}{2}\right) \sinh\left(\frac{t}{2}\right) \sinh\left(\frac{s+t}{2}\right)\right)},
\end{eqnarray*}
which is equal to the function $S$ given in  \eqref{Sfunction}.
Hence
\begin{equation*}
\int_0^\infty \int_0^{2\pi} b_2''(\xi,-1) \frac{1}{(2\pi)^2 \tau_2} rdr d\theta =\frac{1}{4\pi \tau_2} S(\nabla_1,\nabla_2)\Big([\delta_1(\log k),\delta_2(\log k)]\Big).
\end{equation*}

 To compute the Ricci operator one could  directly compute the second term $a_2$ in the asymptotic expansion of the heat kernel of the de Rham Laplacian    using pseudodifferential calculus. This would involve substantial amount of computer aided symbolic calculations. Instead we have observed that our Laplcian is a first order perturbation of the Laplacian in  \cite{Connes-Moscovici2014, Fathizadeh-Khalkhali2013} (cf. formula  \eqref{dallap-deRhamlap}). 
 With this observation, our calculations were reduced to those  carried in this section.

\addcontentsline{toc}{section}{References}
\def\polhk#1{\setbox0=\hbox{#1}{\ooalign{\hidewidth
  \lower1.5ex\hbox{`}\hidewidth\crcr\unhbox0}}}
\providecommand{\bysame}{\leavevmode\hbox to3em{\hrulefill}\thinspace}
\providecommand{\MR}{\relax\ifhmode\unskip\space\fi MR }
\providecommand{\MRhref}[2]{%
  \href{http://www.ams.org/mathscinet-getitem?mr=#1}{#2}
}
\providecommand{\href}[2]{#2}

\Addresses
\end{document}